\documentclass{article}
\usepackage{amssymb}
\usepackage{amsfonts}
\usepackage{mathrsfs}
\usepackage{stmaryrd}
\usepackage{amsmath}
\usepackage{latexsym,amssymb,amsfonts,amsbsy, amsthm}

\newtheorem{theorem}{Theorem}[section]
\newtheorem{lemma}[theorem]{Lemma}
\newtheorem{corollary}[theorem]{Corollary}
\newtheorem{remark}[theorem]{Remark}
\newtheorem{proposition}[theorem]{Proposition}
\newtheorem{definition}[theorem]{Definition}
\numberwithin{equation}{section}

\begin{document}

\title{\textbf{Global structure and regularity of solutions to the Eikonal equation} }
\author{Tian-Hong Li\footnote{The first author is
supported by National Natural Science Foundation of China  under contract 10931007 and CAS NCMIS Y62911CZZ1 and Y629096ZZ3.}\\
Academy of Mathematics and Systems Science \\and Hua
Loo-Keng Key
 Laboratory of Mathematics\\
 Chinese Academy of Sciences\\ Beijing 100190, P. R. China\vspace{0.1cm}\\
 JingHua Wang\footnote{The second author is
supported by National Natural Science Foundation of China  under contract 11471332.}\; \\
 Academy of Mathematics and System
Sciences\\ Chinese Academy of Sciences\\Beijing 100190, P. R.
China\vspace{0.1cm}\\
HaiRui Wen\\Department of Mathematics\\Beijing Institute of Technology\\Beijing 100081,  P. R. China
}\date{} \maketitle \vspace{0.4cm}

\begin{center}{\bf Abstract}\end{center}

{\bfseries Abstract.} The time dependent Eikonal equation  is a Hamilton-Jacobi equation with  Hamiltonian $H(P)=|P|$, which is  not  strictly convex nor smooth.  The regularizing effect of Hamiltonian for the Eikonal equation is much weaker than that of strictly convex  Hamiltonians, therefore  leading to  new phenomena such as the appearance of "contact discontinuity". 
In this paper, 
 we study the set of singularity points of solutions in the upper half space for $C^1$ or $C^2$ initial data, with emphasis on the countability of connected components of the set.
  The regularity of solutions in the complement of the set of  singularity points is also obtained.
 \vspace{0.5cm}



\section{Introduction}
Consider the initial value problem for the time dependent Eikonal  equation
 \begin{equation}\label{1.2} \left\{\begin{array}{rl}
u_{t}+|\nabla u|=0& \quad \mbox{in}\  \mathbb{R}^{n}\times(0,\infty),\\
u=g& \quad \mbox{on}\  \mathbb{R}^{n}\times\{t=0\}.
\end{array}
\right.
\end{equation}
where  $|\cdot|$ is the Euclidean norm defined as
$$|X|=(\Sigma_{i=1}^nx_i^2)^{\frac{1}{2}}$$
for  an $n$-dimensional vector $X$ with components $x_i$, $i=1...n$. $\nabla u$ is the gradient vector $\nabla u=(\frac{\partial u}{\partial x_1}, ..., \frac{\partial u}{\partial x_n})$. (\ref{1.2}) is  a special Hamilton-Jacobi equation, $u_{t}+ H(x, \nabla u)=0$, with $H(x, \nabla u)=|\nabla u|$. This Hamiltonian is not strictly convex, nor superlinear nor smooth. 
 Yet here we consider the  case where $H$ does not depend on $x$, for which many phenomena are more transparent.

 It is worth mentioning that the  Eikonal equation is regarded as a very  important equation in geometric
optics.  It is derivable from Maxwell's equations of electromagnetics by WKB methods, and provides a link between physical (wave) optics and geometric (ray) optics.

The  Eikonal equation (\ref{1.2}) also appears as a typical  Mayer problem in optimal control. We refer the reader to  the comprehensive  and beautiful book \cite{C} of Cannarsa and Sinestrari and  references therein for the framework of  Mayer problem and related Hamiton-Jacobi equations.

No matter how smooth initial data are, solutions to  Hamilton-Jacobi equations are not assured to be in $C^1$. In general, there are nondifferentiable points of solutions  in ${\mathbb R^n}\times \{t>0\}$.  Hence it is interesting to study   sets of singularity points, including  the set of
nondifferentiable points $\Sigma$, its closure $\overline \Sigma$, and the set of nondifferentiable points and generating points of nondifferentiable points $\Sigma\cup T_1$.

 For the Eikonal equation  (\ref{1.2}),
the Legendre transform of $H(P)=|P|$ is
\begin{eqnarray*}
H^*(Q)&=&
\left\{
\begin{array}{ll}
0, & \text{when $|Q|\le 1$}; \\
\infty, & \text{when $|Q|> 1$},
\end{array}
\right.
\end{eqnarray*}
 to which the solution given by the Hopf-Lax formula  has the following form  (see\cite{L.C.}\cite{Bl}\cite{C}\cite{Lion}):
\begin{equation}\label{HL}u(x, t)=\min_{|x-y|\le t}g(y).
\end{equation}
This solution is the unique viscosity solution to the Eikonal equation.
 It is notable that the range of minimizers for $(x, t)$ is a closed ball $\subset {\mathbb R^n}$, not the whole ${\mathbb R^n}$ space. 
   So the criterion of local minimum in ${\mathbb R^n}$ cannot be applied here. 
    Different techniques and new methods are needed in our analysis. 

Instead of general initial data, we restrict to $C^1$ or $C^2$ initial data so that we can study solution properties in the framework of the space of semiconcave functions. 
When  initial data is $C^1$, the solution to (\ref{1.2}) is much less regular than the case with $C^2$ initial data.

In order to better understand the set $\Sigma$   of nondifferentiable points of solutions to the Eikonal equation, we  establish a more complete criterion (Theorem \ref{diff}) to characterize nondifferentiable points in ${\mathbb R^n}\times \{t>0\}$. 
 This criterion is different from the criterion for the case of strictly convex, superlinear and smooth Hamiltonians.

 By  concrete constructions, some  structures of $\Sigma\cup T_1$ or $\Sigma$ or $\overline\Sigma$  are revealed, 
  where $T_1$ is, roughly speaking, the set of differentiable generating points of nondifferentiable points. 
 For $C^1$ initial data, both $\Sigma\cup T_1$ and $\Sigma$ are composed of at most countable connected components or almost  connected components (Theorem \ref{thm1} and Theorem \ref{Sigma}). The definition of an almost connected set can be found in Definition \ref{almost}. 
 For the closure $\overline\Sigma$, we are unable to obtain the countability of connected components, and only
 under certain additional condition, $\overline\Sigma$ is proved to be composed of at most countable connected components (Theorem \ref{overSigma}).
 For $C^2$ initial data, we obtain a much better result: $\Sigma\cup T_1$  is composed of at most countable path connected components (Theorem \ref{path connected}).

For the regularity property, it is proved that for $C^1$ initial data, on the complement of $\Sigma$, the solution $u$ is $C^1$ (Proposition \ref{C1}). For $C^k$ ($k\ge 2$) initial data, a surprise appears,  even outside $\overline\Sigma$, $C^k$ regularity is no more always true. We find a subtle region $P^0$ such that, sometimes only in the complement of $\Sigma\cup T_1\cup P^0$, 
$u$ is $C^k$ (Theorem \ref{C2}). 

 One    difficulty for the Eikonal equation lies in that the Hamiltonian $H$ is not differentiable at $0$. 
 By defining the first and  second termination points of a characteristic from $y$ with $Dg(y)\neq 0$, and finding the close relation between termination points of characteristics from points with $Dg\neq 0$ and nondifferentiable points of solutions, 
 we are able to establish some global structure of  sets of singularity points. For $C^2$ initial data, by use of curvature, more regular properties of termination points are obtained, and several special eigenvectors of Jacobian matrix of a characteristic are identified, so we obtain satisfactory regularity results.

To avoid a long introduction, we defer the discussion of the new phenomena of the Eikonal equation and previous works in section 2.

This paper is organized as follows: In section 2, we list new phenomena of the Eikonal equation 
and previous works.  In section 3, in the framework
that initial data is in $C^1$,  characteristics are defined first,
the relation between characteristics and optimality is discussed.
Then we define first and second termination points. A criterion
that determines points at which $u$ is differentiable or not differentiable is completely established. Many
properties of termination points are  studied.  We find some structure of sets of singularity points different with the structure for the case of smooth, superlinear and strictly convex Hamiltonian,  
and prove that $\Sigma\cup T_1$ is  composed of at most countable connected
components or almost connected components. The same property holds for $\Sigma$. We also discuss $\overline \Sigma$. 
In the complement of $\Sigma$, $u$ is proved to be in $C^1$. In section
4, When the initial data is in $C^2$,  the global structure of $\Sigma\cup T_1$ is considered. 
Better properties are obtained. 
$\Sigma\cup T_1$ is proved to be composed of
at most countable path connected components. 
In section 5,  regularity properties for $C^k$ ($k\ge 2$) initial data are considered. We study $T_1$ and prove $T_1=\Gamma\setminus \Sigma$, where $\Gamma$ is the set of conjugate points (see \cite{C}). Furthermore, we find a subtle region $P^0$ and prove that on ${\mathbb R^n}\times \{t>0\}\setminus (\Sigma\cup T_1\cup P^0)$,  $u$ inherits the regularity of the initial data.
\section{New phenomena and previous works}
In this section,   new phenomena of the Eikonal equation compared with the case of smooth, superlinear and strictly convex Hamiltonians 
 and  previous works are listed.
\subsection{New phenomena for the Eikonal equation}
There are many new phenomena  for the solutions to the Eikonal
equation. For instances,

1. The appearance of "contact discontinuities" \cite{C}.  

  2. For any point with $Dg=0$ at $t=0$,  there might be a a family of characteristics with direction $|P|\le 1$ from it. The phenomena of "contact discontinuity" is caused by those families of characteristics.

  3. 
The set of nondifferentiable points can terminate at finite time. 

4. 
 A point having more than one minimizer does not imply that it is a nondifferentiable point of solutions.
 
 5. For $C^2$ initial data,  nondifferentiable points could appear right after $t=0$.

6.  When $H$ is strictly convex, superlinear and smooth,  there is a one- to- one correspondence between $R_i$ and $S_i$, where $R_i$ is a path connected component  of a set $\tilde R\subset {\mathbb R^n}\times \{t=0\}$,  $S_i$ is a  component  of $\Sigma\cup T_1$.  In another words, let   $\tilde R_i$ be the region that is composed of effective characteristics from points in $R_i$, then  each $\tilde R_i$ contains one and only one  $S_i$.  
For the Eikonal equation, we find out that in general, it does not have this one- to- one structure.

 7. 
We no longer have the strong
regularizing effect which ensures the
solutions $u(x, t)$  to be semiconcave when
the initial data is only Lipschitz (See \cite{C}) or lower semi-continuous (see \cite{Lion}). This has been found already in \cite{C} since  Theorem 7.2.3 therein tells us that $u(x,
t)$ can only be ensured to be Lipschitz if the initial data is
Lipschitz. The singularity in the initial data will propagate along
characteristics.

 8. When $H$ is strictly convex, superlinear and smooth, and initial data is $C^2$, $\Sigma\cup T_1=\overline \Sigma$. While for the Eikonal equation, $\Sigma\cup T_1\subset\overline\Sigma$ and $\Sigma\cup T_1\cup P^0\supset \overline\Sigma$.

  9. When $H$ is strictly convex, superlinear and smooth, and initial data is $C^k$, $k\ge 2$, for any point in  the complement of $\overline \Sigma $, there is a neighborhood $U$ of it, $u$ is $C^k(U)$.  While for the Eikonal equation, sometimes only in  the complement of $\Sigma\cup T_1\cup P^0 $, for any point,  there is a neighborhood $U$ of it, $u$ is $C^k(U)$.

  10. When $H$ is strictly convex, superlinear and smooth,  there is a one- to- one correspondence between $D^*u(x, t)$, the set of all reachable gradients of $u$ at $(x, t)$,   and the set of the minimizers for $(x, t)$; i.e. the set of effective characteristics. While for the Eikonal equation,  it is false. 
  We find out that there is  a new one- to- one correspondence, it is  between  $D^*u(x, t)$ and the set of  values of $Dg$
at minimizers for $(x, t)$.

\subsection{Previous works}
 There are  many important works about the Hamilton-Jacobi equation. Here we only list  some very closely related works. There is  some relation between the Hamilton-Jacobi equation and  conservation laws: 
 the gradient of solutions to  the Eikonal equation $$v=(v_1, v_2, v_3,...,v_n) = \nabla u$$ satisfies the following multi-dimensional  conservation laws: \begin{equation}\label{c}(v_i)_t+H(v)_{x_i}=0, \quad i=1,  2, ..., n.\end{equation} It is well known that the study of multi-dimensional conservation laws is very difficult. As  multidimensional conservation laws, (\ref{c}) has a simple structure, 
  but it provides a picture of how solutions to multi-dimensional conservation laws can possibly behave.    

The solution to the Hamilton-Jacobi equation with convex Hamiltonian  is given by Hopf-Lax formula.  It is known  that the solution given by the Hopf-Lax formula is  the unique  viscosity solution of the Hamilton-Jacobi equation in many circumstances. The concept of
viscosity solutions is introduced by  Crandall and  Lions   (\cite{Lion}\cite{Cr1}\cite{Cr2})
to single out the  physically relevant  solution and therefore guarantee the uniqueness.
For general $H$, Alvarez, Barron, Ishii \cite{A} proved that
the solution given by the Hopf-Lax formula
is  the unique lower semi-continuous viscosity solution, when g
is lower semi-continuous, possibly with value $\infty$ and satisfies $
g(y)\ge -M(|y|+1)$ and H is merely
continuous, finite and convex. From the viewpoint of optimal control,  it is
also proved in \cite{C} that when $g$ is locally Lipschitz and $H$ is convex and  locally Lipschitz, the solution given by Hopf-Lax formula is the unique  viscosity solution. Hence solutions to the Eikonal equation defined by the  Hopf-Lax formula  are unique viscosity solutions. When $H$ is superlinear,  see \cite{L.C., L.C1., K.S.N, St} and  references  therein. When $H$ is superlinear, for a conservation law, the corresponding version of viscosity solutions is closely related to entropy solutions ( see \cite{La, L.C1., Ol} and reference  therein). There are numerous studies about regularity of the viscosity solutions
 (for instances, see  \cite{C}).

When $H$ is strictly convex, superlinear and smooth, there are numerous results about the set of
nondifferentiable points $\Sigma$ and  the regularity of solutions in the set of  differentiable points. When $n=1$, Schaeffer \cite{Scha} proved that,   generically, for initial data having derivatives in the rapidly decreasing space, $\Sigma$ consists of finite number of curves and solutions are $C^\infty$ outside $\overline \Sigma$. Here generic property means that there exists an intersection of  countable open dense sets in the function class of initial data such that this property holds. For
 $C^{4, 1} (C^k, k\ge 5)$ initial data, Li \cite{Li2} proved that, generically,  in any compact set in upper half plane ${\mathbb R}\times\{t > 0\}$, $\Sigma$ consists of finite number of curves and solutions are $C^4 (C^k, k\ge 5)$ outside $\overline \Sigma$.  But this is not a generic property for initial data in weaker class. For $C^1$ initial data, it is proved in  \cite{Li2}  that,  generically, the set of
nondifferentiable points $\Sigma$ is dense in ${\mathbb R}\times\{t > 0\}$; that is, $\overline\Sigma={\mathbb R}\times\{t\ge 0\}$. 
 For Lipschitz initial data and initial data in the class where the derivative functions are of  bounded variations, rarefaction waves might appear, their generic properties are that  $\Sigma$ is dense in the exterior of rarefaction waves \cite{Li2}. 

 When initial data is Lipschitz,  Li \cite{Li2} proved that
the set of  curves at which the points are nondifferentiable and starting points of those curves has $L_2$ measure 0.  Furthermore,  for $C^2$ initial data,  he showed  that
it is a closed set; i.e. $\Sigma\cup St=\overline \Sigma$, $St$ being the set of  starting points of those "shock" curves. So combining above,  $\overline \Sigma$ is a set of  $L_2$ measure 0. Recall from above, the generic property for the case of $C^1$ initial data is $\overline \Sigma={\mathbb R}\times\{t\ge 0\}$,  which implies that the set of nondifferentiable points  for the $C^2$ case might be much smaller than  the $C^1$ case.

When $n>1$, and solutions $u$ are semiconcave,  using geometric measure theory, it is proved that (find more details  in \cite{C}) the set of nondifferentiable points $\Sigma$ is countably $n$-rectifiable.
For convenience of readers, we recall the definition of  $n$-rectifiable set here. 
Definition: Let $C\subset {\mathbb R^{n+1}}$.
 \begin{itemize}
 \item i) $C$ is called an $n$-rectifiable set if there exists a
Lipschitz continuous function $f:{\mathbb R^{n}}\rightarrow {\mathbb R^{n+1}}$ such that
$C\subset f({\mathbb R^{n}})$.
\item ii) $C$ is called a countably $n$-rectifiable set if it is the union of a
countable family of $n$-rectifiable sets.
 \item iii) $C$ is called a countably $H^n$-rectifiable set if
there exists a countably $n$-rectifiable set  $E\subset {\mathbb R^{n+1}}$ such that $H^n(C\setminus E)=0$.\end{itemize}
 In \cite{Flem}, when initial data is $C^k$, $k\ge 2$, it is shown that $\Sigma\cup\Gamma$ is a closed set; i.e. $\Sigma\cup\Gamma=\overline \Sigma$,  $\Gamma$ being the set of conjugate points, which are the "generating" points of nondifferentiable points as shown in \cite{C}.   The solutions inherit the regularity of initial data in the complement of this  closed set (also see \cite{ZTW}\cite{C}). Furthermore, $\Sigma\setminus\Gamma$ is contained in a locally finite union of n-dimensional smooth surfaces. In \cite{C}, $\Gamma$ is a countably $H^n$-rectifiable set, hence $\overline \Sigma$ is a countably $H^n$-rectifiable set. In \cite{Flem}\cite{CMS}, the Hausdorff dimension of the set $\Gamma\setminus\Sigma$ is given.

 About the propagation of singularity, in \cite{C}, when the solution $u$ is semiconcave with linear modulus (this function class is stronger than $C^1$ class, weaker than $C^2$ class), from any nondifferentiable point $z_0$, there exists a Lipschitz curve $z(s)\subset\Sigma$ starting from $z_0$. This curve is a generalized characteristic defined by Dafermos \cite{Da}. Furthermore, from $z_0$, there is a $\nu-$ dimensional Lipschitz surface $\in \Sigma$, $\nu$ is the dimension of the normal cone to $D^+u(z_0)$; i.e. $\nu=n+1-dim D^+u(z_0)$. When initial data is $C^2$, from any point $z_0$ in $\Gamma$, there also exists a Lipschitz curve $\subset\Sigma\cup\Gamma$ starting from $z_0$. When initial data is weaker,  for any point $z_0$ in $\Sigma$,   it is a convergence point of  points in $\Sigma$ at later time (see \cite{C}\cite{CMS}).

 From the topological view, when $n=1$, in \cite{Li}, when initial data is Lipschitz, there is a one- to- one correspondence between path connected components of a set of ${\mathbb R}\times \{t=0\} $ and path connected components of the set of singularities. A singularity here means a nondifferentiable point or a starting points of the curves at which the points are nondiffernetiable.   
  When $n>1$, when the initial data is in $C^2$ and
bounded, it is   shown in \cite{ZTW} that  there is a one- to- one correspondence between path connected components of a set of ${\mathbb R^n}\times \{t=0\} $ and path connected components of   $\overline \Sigma  \;(i.e., \Sigma\cup \Gamma)$ by proving the singularity map is continuous. 
 When
the initial data is in $C^2$ and unbounded, in \cite{LL},   a
similar result is obtained with richer phenomena.

When the initial data is weaker than $C^2$,  a conjugate point can not be defined. Following \cite{Ol}\cite{Li}\cite{ZTW},   the set of termination points of characteristics such as those points have unique minimizer is defined, denoted as $T_1$.  In \cite{ZTW} it is shown that when initial data is $C^2$, $T_1=\Gamma\setminus \Sigma$. Hence $T_1$ is a generalization of the concept of $\Gamma\setminus \Sigma$ for weaker class of initial data. When $n=1$, a point in $T_1$ must be a starting point of a "shock" wave \cite{Li}. When $n>1$, the  lower  second derivatives of  $u$ at a point in $T_1$ must be $-\infty$  \cite{LL} (also see \cite{Lin}\cite{Li}), while for any  point in the complement of $\Sigma\cup T_1$, the lower  second derivatives of  $u$ at it is bounded.  A point in $T_1$ is a starting point of a sequence of nondifferentiable points at later time \cite{LT}.

When initial data is Lipschitz or
lower
semi-continuous, in \cite{LT} it is shown that there is a one- to- one correspondence between path connected components of a set of ${\mathbb R^n}\times \{t=0\} $ and  connected (or almost connected ) components  of  $\Sigma\cup T_1$.
When initial data is weaker than $C^2$,  in general, $\Sigma\cup T_1\neq \overline \Sigma$, it is a proper subset of $\overline \Sigma$.

For the Mayer problem where the Eikonal equation (\ref{1.2}) is a typical case, the tool of Pontryagin maximum priciple is used instead of usual Euler-Lagrange condition. In \cite{C},  when initial data is $C^1$ and semiconcave with linear modulus, for any nondifferentiable point $z_0$, similar to the case of strictly convex, superlinear and smooth Hamiltonians,   there still exists a Lipschitz curve $z(s)\subset \Sigma$ starting from $z_0$. When $0\neq D^* u(x, t)$, there is a one-to-one correspondence between $D^* u(x, t)$ and the effective characteristics passing through point $(x, t)$.

\section{$C^1$ initial data}
\subsection{Preliminary}
For convenience, we denote  the open and closed ball centered at $x$ with radius $t$ by $$B_t(x)=\{y|\;\;|y-x|<t\},\;\overline B_t(x)=\{y|\;\;|y-x|\le t\}.$$ If $u(x, t)=g(y)$ and $|y-x|\le t$, then  $y$ is called a minimizer for $(x, t)$. 

 By Proposition 2.1.2 in \cite{C} that if  $g$ is in
$C^1$, then $g$ is locally semiconcave and  by Remark 7.2.9 in \cite{C} that
if initial data $g$ is semiconcave, then $u$ is semiconcave,  it implies that
if $g$ is in $C^1$, then $u$ is locally semiconcave. Since we want to
study  solutions $u$ in the framework of semiconcave functions,
 for convenience,   our initial data will be assumed in $C^1$.

First in this section, characteristics are defined, which play
important role in studying the regularity and global structure of
solutions to the Eikonal equation.

For  general Hamilton-Jacobi equations, given $y_0\in {\mathbb R^n}$,
the characteristic from $y_0$ can be written as$$C=
\left\{(x,t)|x=y_0+tDH(Dg(y_0)),
  t>0\right\}. $$
  For the Eikonal equation, $H(p)=|p|$, it is easy to deduce
  $DH(p)=\frac{p}{|p|}$, $p\neq 0$; when $p=0$, we use
  subdifferential $$D^-H(p)|_{p=0}=D^-|p||_{p=0}=\{P\in
  {\mathbb R^n}||P|\le 1\},$$ recall the definition of
  subdifferential in Definition \ref{sub} later in this paper.
Hence the characteristics of the Eikonal equation can be defined as
\begin{definition} 1) When $Dg(y_0)\neq 0$,
 the characteristic from $y_0$  is \\$C= \left\{(x,t)|x=y_0+t
 \frac{Dg(y_0)}{|Dg(y_0)|},
  t>0\right\}$.

  2) When $Dg(y_0)= 0$, there is a family of characteristics from
  $y_0$, $C=\left\{(x,t)|x=y_0+tP,
  t>0\right\}$, where $|P|\le 1$.
\end{definition}
 Next we prove that characteristics have the optimality, that is,
for any $y_0\in {\mathbb R^n}$, only points on the characteristic
from $y_0$ are possible to have $y_0$ as their minimizer; points on
any other lines  from $y_0$ do not have $y_0$ as their minimizer.
\begin{proposition}\label{NC}For any $y_0\in {\mathbb R^n}$, if
a straight line $l$ from $y_0$ is not the characteristic from $y_0$,
then any point $(x, t)$ on $l$ does not have $y_0$ as its minimizer.
\end{proposition}
\begin{proof}First we look at the case when $Dg(y_0)\neq 0$.
 The characteristic is $C= \left\{(x,t)|x=y_0+t\frac{Dg(y_0)}{|Dg(y_0)|},
  t>0\right\}$. If $l$ is not the characteristic, then
  $ l=\left\{(x,t)|x=y_0+Pt, t>0\right\}$, where
  $P\neq \frac{Dg(y_0)}{|Dg(y_0)|}$.

1. $|P|>1$.

For any $t$, $x=y_0+Pt$,  we have $|x-y_0|=|Pt|>t$. Then $y_0$ is not in $\overline B_t(x)=\{y||x-y|\le t\}$.

2. $|P|<1$.

For any $t$, $x=y_0+Pt$,  we have $|x-y_0|=|Pt|<t$. $y_0$ is in the interior of $\overline B_t(x)$. Since $Dg(y_0)\neq 0$, $y_0$ is not minimizer for $(x,t)$ in $\overline B_t(x)$.

3. $|P|=1$, $P\neq \frac{Dg(y_0)}{|Dg(y_0)|}$.

For any $t$, $x=y_0+Pt$, we have that for any $y\in \overline B_t(x)$, $(y-y_0)\cdot P\ge 0$. Since  $P\neq \frac{Dg(y_0)}{|Dg(y_0)|}$ and $|P|=1$, we have that the direction of $P$ and $Dg(y_0)$ are different. Hence there must exist some $y_1\in \overline B_t(x)$ sufficiently close to $y_0$ such that $(y_1-y_0)\cdot Dg(y_0)< 0$. Then $$g(y_1)=g(y_0)+Dg(y_0)\cdot (y_1-y_0)+o(|y_1-y_0|)<g(y_0).$$
Therefore $y_0$ is not a minimizer for $(x,t)$.

We look at the case when $Dg(y_0)=0$. If $|P|>1$, the reason is  same as the case  when $Dg(y_0)\neq 0$ above.
\end{proof}
 When $Dg(y_0)\neq 0$, we can define two times of
the characteristic and two points of the
characteristic.\begin{definition} Let $\overline t_s$ be the first
termination time of $C$, where\\
$C:\left\{(x,t)|x=y_0+t\frac{Dg(y_0)}{|Dg(y_0)|}, t>0\right\}$,
\begin{equation}\label{bart_s}\overline t_s(y_0)=\sup\{t\mid
y_0\; \mbox{is  unique minimizer for}\,(x,t), (x, t)\in C\}.
\end{equation} Let $t_s(y_0)$ be the second
termination time of $C$:
\begin{equation}\label{t_s}t_s(y_0)=\inf\{t\mid
y_0\; \mbox{is not a minimizer for}\,(x,t), (x, t)\in C\}.
\end{equation}
Accordingly, we can define termination points when termination times
are finite:

1)If $\overline t_s(y_0)<\infty$, then $(\overline x_s(y_0),
 \overline t_s(y_0))$ is called the first termination point of $C$,
 where $\overline x_s(y_0)=y_0+\frac{Dg(y_0)}{|Dg(y_0)|}
 \overline t_s(y_0)$.

 2)If $t_s(y_0)<\infty$, then $( x_s(y_0),  t_s(y_0))$ is called the
second termination point of $C$, where
$x_s(y_0)=y_0+\frac{Dg(y_0)}{|Dg(y_0)|}t_s(y_0)$.

\end{definition}
It is easy to see that
\begin{equation}\overline t_s(y_0)\le t_s(y_0).
\end{equation}
Similarly we can define termination times and termination points of
the characteristic from $y_0$ when $Dg(y_0)=0$. When characteristic
has direction $P$ with $|P|\le 1$, we can write $t_s$ as $t_s(y_0,
P)$, $\overline t_s$ as $\overline t_s(y_0, P)$.

 With the help of second termination time,  effective characteristics
  can be defined . \begin{definition}1) When $Dg(y_0)\neq
0$,
 the effective characteristic is $\bar C= C\cap \{t\le t_s(y_0)\}$.

2) When $Dg(y_0)=0$ and  $C=\left\{(x,t)|x=y_0+tP,
  t>0\right\}$, the effective characteristic is
  $\bar C= C\cap \{t\le t_s(y_0, P)\}$. \end{definition}
   The definition of
characteristic from $y_0$ is defined  by only using $Dg(y_0)$. While
the definition of effective characteristic from $y_0$ is related to
the optimality.

Define multifunction \begin{equation}\label{L(x, t)}L(x, t)=\{y|u(x, t)=g(y), \,y\in \overline B_t(x)\}.
\end{equation}That is, $L(x, t)$ is the set of minimizers for $(x, t)$.
It is easy to see that   \begin{lemma}$y_0\in L(\overline x_s(y_0), \overline t_s(y_0))$ and $y_0\in L(x_s(y_0),  t_s(y_0))$.
\end{lemma}
\subsection{Criterion for nondifferentiable points}
Fix $t$, we denote the directional derivatives of $u(x, t)$ along $l\in {\mathbb R^n}$ by $\partial_lu(x, t)$. The following  lemma holds:
\begin{lemma}\label{direction}Suppose that $g$ is semiconcave, then we have
\begin{equation}\partial_lu(x, t)=\min_{y\in L(x, t)}\partial_lg(y).
\end{equation}
\end{lemma}  This lemma has similar spirit as Lemma 9.2 in \cite{V}.
\begin{proof}Assume unit vector $l\in {\mathbb R}^n\setminus 0$, from the definition of directional derivative, there
exist $\{\alpha_k\}$, $\{\beta_k\}$ such that
\begin{equation}\partial_l^+u(x, t)=\limsup_{\delta\rightarrow 0^+}
\frac{u(x+l\delta, t)-u(x, t)}{\delta}=\lim_{\alpha_k\rightarrow 0^+}
\frac{u(x+l\alpha_k, t)-u(x, t)}{\alpha_k},
\end{equation}
\begin{equation}\partial_l^-u(x, t)=\liminf_{\delta\rightarrow 0^+}
\frac{u(x+l\delta, t)-u(x, t)}{\delta}=\lim_{\beta_k\rightarrow 0^+}
\frac{u(x+l\beta_k, t)-u(x, t)}{\beta_k}
\end{equation}
Set $u(x, t)=g(y)$, $\forall y\in L(x, t)$ and notice that
$u(x+l\alpha_k, t)\le g(y+l\alpha_k)$, we have
$$\frac{u(x+l\alpha_k, t)-u(x, t)}{\alpha_k}\le
\frac{g(y+l\alpha_k)-g(y)}{\alpha_k}.$$ Theorem 3.2.1 in \cite{C}
tell us that if $g$ is semiconcave, then for any $z$ and $l\in
{\mathbb R^n}$, the directional derivative $\partial_lg(z)$ exists.
Hence let $\alpha_k\rightarrow 0^+$, the right side term approaches
$\partial_lg(y)$. Thus
\begin{equation}\partial_l^+u(x, t)\le \inf_{y\in L(x, t)} \partial_lg(y).
\end{equation}
Denote $u(x+l\beta_k, t)=g(y_k)$, $\forall y_k\in  L(x+l\beta_k,
t)$, and notice that $|x-y_k+l\beta_k|\le t$, we have
\begin{equation}\label{2.8}\frac{u(x+l\beta_k, t)-u(x, t)}{\beta_k}\ge \frac{g(y_k)-g(y_k- l\beta_k)}{\beta_k}.
\end{equation}The right side
$$ \frac{g(y_k)-g(y_k- l\beta_k)}{\beta_k}= \frac{g(y_k-l\beta_k+l\beta_k)-g(y_k- l\beta_k)}{\beta_k}.$$ There is a subsequence of $\{y_k\}$, still denoted as $\{y_k\}$, converging to some point $\bar y\in L(x, t)$. When $\beta_k\rightarrow 0^+$,   $y_k-l\beta_k\rightarrow \bar y$. By Definition 3.1.11 in \cite{C} about the  generalized lower derivative $g_-^0(z, l)$, which is defined as $\liminf_{h\rightarrow 0^+, z'\rightarrow z}\frac{g(z'+hl)-g(z')}{h}$  and Theorem 3.2.1 in \cite{C} that if $g$ is semiconcave, then $\partial_lg(z)=g_-^0(z, l)$, now let $\beta_k\rightarrow 0^+$ in (\ref{2.8}), we have
\begin{eqnarray*}\partial_l^-u(x, t)&\ge& \liminf_{\beta_k\rightarrow 0^+,\;y_k-l\beta_k\rightarrow \bar y}\frac{g(y_k-l\beta_k+l\beta_k)-u(y_k- l\beta_k)}{\beta_k}\\&\ge& g_-^0(\bar y, l)=\partial_lg(\bar y)\ge \inf_{y\in L(x, t)}\partial_lg( y).
\end{eqnarray*}
Therefore $$\partial_l^+u(x, t)=\partial_l^-u(x, t)=\inf_{y\in L(x, t)}\partial_lg( y).$$
Since $L(x, t)$ is a closed set, $\inf_{y\in L(x, t)}\partial_lg( y)=\min_{y\in L(x, t)}\partial_lg( y)$. Therefore the proof is complete.
\end{proof}
From the proof above, we can see that if $g$ is assumed to be in $C^1$, it will be easier to prove.

The above lemma tells us the property of directional
derivatives of $u(x, t)$, but it is well known that the existence of
directional derivatives of all direction does not imply the
differentiability of $u$ at this point. Even for any two functions
$f$ and $h$, $h$ is differentiable at $(x, t)$ and
$\partial_lf(x,t)=\partial_lh(x,t)$ for any direction $l$, it does
not imply that $f$ is differentiable at $(x, t)$. To prove
differentiability, first recall some definition about classical
derivatives. $f$ is said Gateaux differentiable at $x$ provided 1)
$\partial_lf(x)$ exists for all direction $l$, 2) and there exists a
(necessarily unique) element $ f'_G(x)$ (called the Gateaux
derivative) that satisfies
\begin{equation}\label{Gat}\partial_lf(x)=<f'_G(x), l>, \end{equation} for
all $l$.(See page 31 in \cite{Cl} ).

Next we have a criterion  about
differentiability of $u(x,t)$ at $(x_0, t_0)$. It uncovers the relation between  differentiability of $u(x,t)$ at $(x_0, t_0)$ and minimizers for $(x_0, t_0)$.   Before giving the
criterion, let us recall some generalized differentials (see \cite{C}), they will be needed  in the proof.
\begin{definition}\label{sub}Given $v: \Omega\rightarrow {\mathbb R}$,
$\Omega$ being an open subset of ${\mathbb R^N}$, and $z\in \Omega$,
the (Frechet) subdifferential of $v$ at $z$ is the set
$$D^-v(z)=\{p\in {\mathbb R^N}|\liminf_{h\rightarrow 0}
\frac{v(z+h)-v(z)-p\cdot h}{|h|}\ge 0\}.$$ Similarly, the (Frechet)
supdifferential of $v$ at $z$ is the set
$$D^+v(z)=\{p\in {\mathbb R^N}|\limsup_{h\rightarrow 0}
\frac{v(z+h)-v(z)-p\cdot h}{|h|}\le 0\}.$$
We say that $p$ is a reachable gradient of $v$ at a point $z\in\Omega$ if there exists $\{z_k\}$ of points at which $v$ is differentiable such that $$z_k\rightarrow z, \qquad Dv(z_k)\rightarrow p.
$$ We denote with $D^*v(z)$ the set of reachable gradients of $v$ at $z$.
\end{definition}

\begin{theorem}\label{diff}Suppose $g(y)\in C^1$, then we have the following
results:

1) If $L(x_0, t_0)$ is a singleton, then $u$ is differentiable at
$(x_0, t_0)$.  And $Du(x_0, t_0)=(\nabla u, u_t)(x_0, t_0)=(Dg(y_0), -|Dg(y_0)|)$, where $y_0\in L(x_0, t_0)$.

2) If $L(x_0, t_0)$ is not a singleton, and there exists at
least one, say, $y_1\in L(x_0, t_0)$ such that $D(y_1)\neq 0$, then
$u$ is nondifferentiable at $(x_0, t_0)$.

 3) If $Dg(y)=0$ for all $y\in L(x_0, t_0)$, then $u$ is differentiable at
$(x_0, t_0)$. And $Du(x_0, t_0)=0$.

\end{theorem}
\begin{remark}For a more general setting \cite{C},
1) is obtained in \cite[Theorem 7.3.14]{C},
2) can be obtained by combining \cite[Theorem 7.3.3]{C},
\cite[Corollary 7.3.7]{C}, and \cite[Theorem 7.3.9]{C}.
 The   result 3) characterizes an essential difference
 between the Eikonal equation and strictly convex, superlinear, smooth Hamiltonians:  a differentiable  point for the Eikonal equation
 can have more than one minimizer. Example 7.2.10 in \cite{C} has   pointed it out. This should be due to the singular point of $H$ at $0$.
   3) also gives the sufficient and necessary condition to a differentiable point which has more than one minimizer.

 In below we give a self-contained proof which
 takes advantage of the simple form of Hopf-Lax formula (\ref{HL})
 for the Eikonal equation (\ref{1.2}).
\end{remark}
\begin{proof}First we prove 1). If $L(x_0, t_0)$ is a singleton,
assume $y_0$ is the unique element in $L(x_0, t_0)$, by Lemma
\ref{direction}, we have $$\partial_lu(x_0, t_0)=\partial_lg(y_0).
$$ Since $g$ is differentiable at $y_0$, we have  $$
\partial_lu(x_0, t_0)=\partial_lg(y_0)=Dg(y_0)\cdot l. $$Recall (\ref{Gat}),
 it implies that
$u$ is Gateaux differentiable  at $(x_0, t_0)$ for $x$ variable.
11.20 d), in page 66 in \cite{Cl}, it  says that  if $f : {\mathbb R}^n
\rightarrow{\mathbb R}$ is Lipschitz near $x$, then at $ x$, Gateaux
differentiability implies differentiability (this is false in
general in infinite dimensional spaces). We have known $u$ is
semiconcave , and semiconcavity implies Lipschitz continuity. Hence
$\nabla u(x_0, t_0)$ exists. By Lemma \ref{direction}, we know that
$\nabla u(x_0, t_0)=Dg(y_0)$.

Next  we want to prove that $ Du(x_0, t_0)$ exists. Since $\nabla
u(x_0, t_0)$ exists, by Proposition 3.1.5 c) in \cite{C}, we have
that $\nabla ^+u(x_0, t_0)=\{\nabla u(x_0, t_0)\}$. By Lemma 3.3.16
in \cite{C}, $\Pi_xD^+u(x_0, t_0)=\nabla ^+u(x_0, t_0)$, where
$\Pi_x:{\mathbb R^{n+1}}\rightarrow {\mathbb R^{n}}$ be the
projection onto the $x-$space. Hence
$D^+u(x_0, t_0)$ is  a subset of the line $ \{(a, \nabla u(x_0,
t_0))|a\in {\mathbb R}\}$. By Theorem 3.3.6 in \cite{C}, $D^+u(x_0,
t_0)=co D^*u(x_0, t_0)$, which is the convex hull of $D^*u(x_0,
t_0)$. It implies that $D^*u(x_0, t_0)$ is also a subset of the line
$ \{(a, \nabla u(x_0, t_0))|a\in {\mathbb R}\}$. Therefore every
element in $D^*u(x_0, t_0)$ has the form $(a, \nabla u(x_0, t_0))$.
It is easy to see that for any $(p_t, p_x) \in D^*u(x_0, t_0)$,   it
satisfies the Eikonal equation $u_t+|\nabla u|=0$. Hence
\begin{equation}\label{pt}p_t+|p_x|=0.\end{equation} Then for any two elements $(a_1,
\nabla u(x_0, t_0))$, $(a_2, \nabla u(x_0, t_0))\in D^*u(x_0, t_0)
$, they satisfy (\ref{pt}). Hence $a_1+|\nabla u(x_0, t_0)|=0$ and
$a_2+|\nabla u(x_0, t_0)|=0$. It implies that $a_1=a_2$. Therefore $
D^*u(x_0, t_0)$ is a singleton.  By Theorem 3.3.6 in \cite{C} again,
we have that
 $D^+u(x_0, t_0)=coD^*u(x_0, t_0)$, it implies that  $D^+u(x_0, t_0)$
 is a singleton. By Proposition 3.3.4
d) in \cite{C} that if $v$ is semiconcave and $D^+v(z)$ is a
singleton, then $v$ is differentiable at $z$, we can conclude that
$u$ is differentiable at $(x_0, t_0)$. Hence at $(x_0, t_0)$, $u$
satisfies Eikonal equation. It implies that $u_t(x_0, t_0)=-|\nabla
u(x_0, t_0)|=-|Dg(y_0)|$. Therefore $Du(x_0, t_0)=(\nabla u(x_0,
t_0), u_t(x_0, t_0))=(Dg(y_0), -|Dg(y_0)| )$.

For 2), if $L(x_0, t_0)$ is not a singleton,  and there exists at
least one, say, $y_1\in L(x_0, t_0)$ such that $D(y_1)\neq 0$, assume $y_2$ is another element in $L(x_0, t_0)$, then we have that $Dg(y_1)\neq Dg(y_2)$. The reason is that 1., if $Dg(y_2)=0$, then $Dg(y_1)\neq Dg(y_2)$. 2., If $Dg(y_2)\neq 0$, then $x_0=y_1+\frac{Dg(y_1)}{|Dg(y_1)|} t_0=y_2+\frac{Dg(y_2)}{|Dg(y_2)|} t_0$, it implies that $\frac{Dg(y_1)}{|Dg(y_1)|}\neq \frac{Dg(y_2)}{|Dg(y_2)|}$, hence \begin{equation}\label{Dg}Dg(y_1)\neq Dg(y_2).\end{equation}
Next we claim that  there must exist some $l$ and some $y_1$, $y_2\in L(x_0, t_0)$ with $Dg(y_1)\neq 0$,   $\partial_lg(y_1)\neq\partial_lg(y_2)$ holds.
 It will be proved  by contradiction argument. Assume that for any direction $l$ and any $y_1$, $y_2\in L(x_0, t_0)$ with $Dg(y_1)\neq 0$, the directional derivatives $\partial_lg(y_1)=\partial_lg(y_2)$ holds, then $Dg(y_1)\cdot l=Dg(y_2)\cdot l$. It implies that $Dg(y_1)=Dg(y_2)$, which contradicts with (\ref{Dg}). Hence the proof of the claim is complete. Assume $\partial_lg(y_1)<\partial_lg(y_2)$. By Lemma \ref{direction}, there exist some $y_3\in L(x_0, t_0)$ such that
\begin{equation}\label{l} \partial_lu(x_0, t_0)=\min_{y\in L(x_0, t_0)}\partial_lg(y)=\partial_lg(y_3)\le \partial_lg(y_1)<\partial_lg(y_2).\end{equation} Hence  \begin{equation}\partial_lg(y_3)<\partial_lg(y_2).\end{equation}
Since $g$ is differentiable at any point $y$, $\partial_lg(y)=-\partial_{-l}g(y)$. We have that $$  -\partial_{-l}g(y_3)<-\partial_{-l}g(y_2).      $$ That is, \begin{equation}\label{-l}\partial_{-l}g(y_3)>\partial_{-l}g(y_2).\end{equation}

Next we observe $\partial_{-l}u(x_0, t_0)$, by (\ref{-l}), (\ref{l}) and there exists some $y_4\in L(x_0, t_0)$ such that
\begin{eqnarray*}\partial_{-l}u(x_0, t_0)&=&\min_{y\in L(x_0, t_0)}\partial_{-l}g(y)=\partial_{-l}g(y_4)\le \partial_{-l}g(y_2)<\partial_{-l}g(y_3)\\&=&-\partial_{l}g(y_3)=-\partial_lu(x_0, t_0).\end{eqnarray*}
It implies that $$\partial_{-l}u(x_0, t_0)<-\partial_lu(x_0, t_0).$$ Therefore $u$ is not differentiable at $(x_0, t_0)$.

Finally we will prove 3). Since $Dg(y)=0$ for all $y\in L(x_0, t_0)$,
then for any direction $l\in{\mathbb R^n}$, by Lemma \ref{direction},
$\partial_lu(x_0, t_0)=\min_{y\in L(x_0, t_0)}\partial_lg(y)=
\min_{y\in L(x_0, t_0)}Dg(y)\cdot l=0$. It implies that
$$\partial_lu(x_0, t_0)=0\cdot l.$$ Hence $u$ is Gateaux differentiable
for $x$ variable at $(x_0, t_0)$. Then similarly to the argument in
the proof of part 1), we can obtain that $u$ is differentiable at
$(x_0, t_0)$. And $Du(x_0, t_0)=0\in {\mathbb R^{n+1}}$.
\end{proof}

We  can define the set \begin{equation}\tilde L(x, t)=\{Dg(y)|y\in
L(x, t)\}.
\end{equation}
With the definition $\tilde L(x, t)$, Theorem \ref{diff} can be
written in the following brief form:
\begin{corollary}Suppose $g(y)\in C^1$. Then $u$ is differentiable
at $(x_0, t_0)$ if and only if $\tilde L(x_0, t_0)$ is a singleton.
\end{corollary}

\subsection{Relation between termination points and nondifferentiable points}

 The next proposition tells us that for the characteristic from $y_0$ with $Dg(y_0)\neq 0$,  when the first termination
 point and the second termination point  coincide, it
 is either a nondifferentiable point
 or a cluster point of nondifferentiable points from above.
 When they do not coincide, every point in the line segment between
 these two termination points is nondifferentiable point.
\begin{proposition}\label{cluster} Let $y_0$ be any point in
${\mathbb R^n}$ such that $Dg(y_0)\neq 0$ and $\overline
t_s(y_0)<\infty$. 1. When $\overline t_s(y_0)=t_s(y_0)$, $(x_s(y_0),
t_s(y_0))$ is either 1) a differentiable
point if $(x_s(y_0),
t_s(y_0))$ has unique minimizer   or 2) a nondifferentiable point if $(x_s(y_0),
t_s(y_0))$ has more than one minimizer. Furthermore,   if it is  a differentiable point, or a nondifferentiable
point such that all minimizers have $Dg\neq 0$, then there exists
$h>0$ with the following property: Given any $\epsilon\in (0, h]$,
there exists $x_{\epsilon}\in B_{\epsilon}( x_s(y_0))$ such that $(
x_{\epsilon}, t_s(y_0)+\epsilon)$ is a nondifferentiable point,
where $\overline B_{\epsilon}( x_s(y_0))$ is the closed ball centered in $
x_s(y_0)$ with radius $\epsilon$. 2. When $\overline t_s(y_0)<
t_s(y_0)$, for any point in $\{(x,
t)|x=y_0+\frac{Dg(y_0)}{|Dg(y_0)|}t, \, t\in [\overline t_s(y_0),
t_s(y_0)]\}$, it is a nondifferentiable point.
\end{proposition}
The proof is quite similar to the proof of Lemma 6.5.1 in \cite{C}.
For the completeness of this paper, we present the proof here.
Before giving the proof,  a lemma is needed here.
\begin{lemma}\label{nabla}If $u$ is differentiable at $(x_0, t_0)$
with $\nabla u(x_0, t_0)\neq 0$,  let $y_1=x_0-\frac{\nabla u(x_0,
t_0)}{|\nabla u(x_0, t_0)|}t_0$, then $y_1$ is the unique minimizer
for $(x_0, t_0)$ and there is a unique effective
characteristic passing through $(x_0, t_0)$.
\end{lemma}
\begin{proof} 
Since $u$ is differentiable at $(x_0, t_0)$, by Theorem \ref{diff},
we have that $\nabla u(x_0, t_0)=Dg(y_0)$, for any $y_0\in L(x_0,
t_0)$. Since $\nabla u(x_0, t_0)\neq 0$, it implies that
$Dg(y_0)\neq 0$. By Theorem \ref{diff} again, $y_0$ is the unique
minimizer for $(x_0, t_0)$. Due to Proposition \ref{NC}, we have
that $x_0=y_0+\frac{Dg(y_0)}{|Dg(y_0)|}t_0$. Let
$y_1=x_0-\frac{\nabla u(x_0, t_0)}{|\nabla u(x_0, t_0)|}t_0$, hence
$y_1=x_0-\frac{\nabla u(x_0, t_0)}{|\nabla u(x_0,
t_0)|}t_0=x_0-\frac{Dg(y_0)}{|Dg(y_0)|}t_0$. It implies that
$y_0=y_1$. Therefore $y_1$ is the unique minimizer for $(x_0, t_0)$.
By Proposition \ref{NC}, $(x_0, t_0)$ is on the effective
characteristic from $y_1$.  If we assume that there is another
effective characteristic from $\bar y$ passing through $(x_0, t_0)$,
then $\bar y$ is also a minimizer for $(x_0, t_0)$. It contradicts
with that $(x_0, t_0)$ has unique minimizer. Hence there is unique
effective characteristic passing through $(x_0, t_0)$, which is from
$y_1$.

\end{proof}
Now we are ready
for the proof of Proposition \ref{cluster}.
\begin{proof}Case 1 when $t_s(y_0)=\overline t_s(y_0)$.

 If $(x_s(y_0), t_s(y_0))$ has more than one minimizer,
 by $Dg(y_0)\neq 0$ and Theorem \ref{diff}, then $(x_s(y_0), t_s(y_0))$ is a nondifferentiable point.
If $(x_s(y_0), t_s(y_0))$ has unique minimizer, by Theorem
\ref{diff}
 again, we know $(x_s(y_0), t_s(y_0))$ is a differentiable point.

 Next the statement of "Furthermore" will be proved  by contradiction arguments. Suppose that there exists some
 $\epsilon\in (0, h]$ such that for any point $x$ in
 $\overline B_\epsilon(x_s(y_0))$, $(x, t_s(y_0)+\epsilon )$ is a differentiable
 point. By Proposition 2.1.2, Remark 7.2.9 and Proposition 3.3.4 in \cite{C},  it
 follows that $u$ is locally semiconcave and  $\nabla u(x, t_s(y_0)+
 \epsilon)$ is continuous on
  $\overline B_\epsilon(x_s(y_0))$.  For any $y\in L(x, t_s(y_0)+
  \epsilon )$ when $\epsilon$ is sufficiently small,  we have
  $y\approx y_1$, where $y_1$ is some minimizer for
  $(x_s(y_0), t_s(y_0))$. Since we assume that
    $(x_s(y_0), t_s(y_0))$
   is  a differentiable point with $Dg(y_0)\neq 0$ or a nondifferentiable
point such that all minimizers have $Dg\neq 0$,
 hence $Dg(y)\neq 0$.  Then by Theorem \ref{diff}, $(x, t_s(y_0)+\epsilon )$  must
 have  unique minimizer $y$ and $\nabla u(x, t_s(y_0)+\epsilon)=Dg(y)\neq
 0$.

 Let $$\Lambda(x)=x_s(y_0)+x-X(t_s(y_0), t_s(y_0)+\epsilon, x,
 \nabla u(x, t_s(y_0)+\epsilon)),$$where $$X(t_s(y_0), t_s(y_0)+\epsilon, x, \nabla
   u(x, t_s(y_0)+\epsilon))=x-\frac{\nabla u(x, t_s(y_0)+\epsilon)}{|\nabla
   u(x, t_s(y_0)+\epsilon)|}\epsilon.$$ By Lemma \ref{nabla}, we know
   that
$X(t_s(y_0), t_s(y_0)+\epsilon,
 x, \nabla u(x, t_s(y_0)+\epsilon))$ is the $x$ component of the point
    at
   time $t_s(y_0)$ on the unique effective characteristic passing through
   $(x, t_s(y_0)+\epsilon)$.   It follows that $\Lambda(x) $ is
   continuous on $\overline B_\epsilon(x_s(y_0))$.

 Furthermore, $$\Lambda(x)-x_s(y_0)=x-X(t_s(y_0), t_s(y_0)+\epsilon, x, \nabla
 u(x, t_s(y_0)+\epsilon))=\frac{\nabla u(x, t_s(y_0)+\epsilon)}{|\nabla
 u(x, t_s(y_0)+\epsilon)|}\epsilon.$$
 Then we have that $|\Lambda(x)-x_s(y_0)|=\epsilon$, therefore
 $\Lambda(x)$ maps $\overline B_\epsilon(x_s(y_0))$ into itself.
 By Brouwer's fixed point theorem, there exists some
 $x_\epsilon\in \overline B_\epsilon(x_s(y_0))$ such that
 $\Lambda(x_\epsilon)=x_\epsilon$. It follows that
 $$x_s(y_0)=X(t_s(y_0), t_s(y_0)+\epsilon, x_\epsilon, \nabla u(x_\epsilon, t_s(y_0)+\epsilon)).$$
 The effective characteristic passing through $(x_\epsilon, t_s(y_0)+\epsilon)$ also
 passes through $(x_s(y_0), t_s(y_0))$.
 Let $y_\epsilon$ be the unique
  minimizer for $(x_\epsilon, t_s(y_0)+\epsilon)$.   If
 $y_\epsilon=y_0$, then it contradicts with the
 definition of $t_s(y_0)$. If $y_0\neq y_\epsilon$, then 
$(x_s(y_0), t_s(y_0))$ have $y_0$ and $y_\epsilon$ as its
minimizers. Since $Dg(y_0)\neq 0$, it implies that $t_s(y_\epsilon)=t_s(y_0)$, which
contradicts with that $y_\epsilon$ is a
  minimizer for $(x_\epsilon, t_s(y_0)+\epsilon)$, which means $t_s(y_\epsilon)\ge t_s(y_0)+\epsilon$.

 Case 2 when $t_s(y_0)>\overline t_s(y_0)$.

 First by the definition of $\overline t_s(y_0)$, for any
 $t\in ( \overline t_s(y_0),  t_s(y_0)]$, the point $(x,t)$ on the
 characteristic from $y_0$ has more than one minimizer. By
  Theorem \ref{diff}, the point $(x, t)$ is nondifferentiable.
 Next we observe the point $(\overline x_s(y_0), \overline
 t_s(y_0))$. For any point $(x,t)$ on the
 characteristic from $y_0$ and $\overline
 t_s(y_0)< t <
 t_s(y_0)$, there exists some $ y_t\in L(x, t)$ such that $g$
 attains
 a local minimum in ${\mathbb R^n}$ at   $ y_t$. Hence we have that
 $Dg(y_t)=0$.  Then there exists some subsequence
 $\{y_{t'}\}$ and $\bar y$ such that $y_{t'}\rightarrow \bar y$ when
 $t'\rightarrow \overline
 t_s(y_0)$ and $\bar y\in L(\overline x_s(y_0), \overline
 t_s(y_0))$. Obviously, $Dg(\bar y)=0$. Hence $y_0\neq \bar y$.
 Therefore $(\overline x_s(y_0), \overline
 t_s(y_0))$ has more than one minimizer with $Dg(y_0)\neq 0$. It implies that it is a
  nondifferentiable point of $u$.
\end{proof}
Although the proof above is quite similar to the proof of Lemma
6.5.1 in \cite{C}, the properties of the case of  $H$ being strictly convex, superlinear and smooth and the
Eikonal equation might not be  the same. For the case of  $H$ being strictly convex and smooth, any nondifferentiable point is  also a cluster point of
nondifferentiable points from above (see Lemma 6.5.1 in \cite{C}),
while for the Eikonal equation (as seen in the proposition above), so far
it  can only be proven that only for those  nondifferentiable points
such that all minimizers have $Dg\neq 0$, they are cluster points of
nondifferentiable points from above. For other nondifferentiable
points,   it is not clear that whether the property holds or not.
The mathematical difficulty is that for the case of  $H$ being strictly convex and smooth,
if $(x, t)$ is a differentiable point, then it implies that $(x, t)$
has unique minimizer; while for the Eikonal equation, if $(x, t)$ is a
differentiable point, it does not imply that $(x, t)$ has unique
minimizer. $(x, t)$ may have more than one minimizer with $Dg=0$.
Therefore if $\nabla u(x, t)=0$, there might be more than one
effective characteristic passing through $(x, t)$. From here, we can
 see clearly one of the differences between   the case of  $H$ being strictly convex and smooth and the
Eikonal equation: $\nabla u=0$ is a trouble for the Eikonal equation,
while it is not a trouble for the case of $H$ being strictly convex and smooth.
\begin{remark}The above proposition tells us that for the
characteristic from $y_0$ with $Dg(y_0)\neq 0$ and $\overline
t_s(y_0)<\infty$, when $\overline t_s(y_0)= t_s(y_0)$, $( x_s, t_s)$
is nondifferentiable point or a cluster point of nondifferentiable
points.  When $\overline t_s(y_0) < t_s(y_0)$, $(\overline x_s,
\overline t_s)$ , $(x_s, t_s)$ and the points in between  are all
nondifferentiable points. That is, every termination point or any
point in between the  first termination point and the second termination point of the characteristic from any $y_0$ with $Dg(y_0)\neq 0$ is
either a nondifferentiable point or a cluster point of
nondifferentiable points.

 Conversely, 1) it is easy to see that nondifferentiable
points must be on characteristics from some $y_0$ with $Dg(y_0)\neq
0$. 2) However, cluster points of nondifferentiable points  are not
necessarily termination points of characteristics. Example 1, if
$\liminf_{y_n\rightarrow y_0} t_s(y_n)<\overline t_s(y_0)$, then
$(x_s(y_n), t_s(y_n))\rightarrow (x,t)$, $t<\overline t_s(y_0)$.
Hence $(x, t)$ is some point before the first termination point of the
characteristic from $y_0$. Example 2, for some $(x, t)$ with $\nabla u(x, t)=0$, it is possible to be  a cluster point of nondifferentiable points as well.
\end{remark}
\subsection{Properties of termination points}
Since termination points of characteristics are closely related to
nondifferentiable points, we further study their properties.

First, we  define  some notations. Since $g$ is $C^1$,
for any $y_0$ such that $Dg(y_0)\neq 0$, by implicit theorem, there exists
a  $n-1$ dimensional $C^1$ surface $H(y_0)\subset {\mathbb R^n}$ such that $y_0\in H(y_0)$ and for any
$y\in H(y_0)$, we have that $g(y)=g(y_0)$. $Dg(y_0)$ is a normal vector
of $H(y_0)$. $H(y_0)$ separates the neighborhood of $y_0$ as
two open sets. One is denoted as $H^-(y_0)$, such that when $y\in H^-(y_0)$,
$g(y)< g(y_0)$. Another is denoted as $H^+(y_0)$, such that when $y\in H^+(y_0)$, $g(y)> g(y_0)$. For clearness, we write it as a definition.
\begin{definition}\label{H} We denote the above surface as $H(y_0)$,
the above two sets as $H^-(y_0)$ and $H^+(y_0)$.
\end{definition}

\begin{proposition}\label{limsup}For any $y_0$ with $Dg(y_0)\neq 0, $\begin{equation}\label{=}
\limsup_{y_n\rightarrow y_0}\overline t_s(y_n)=\limsup_{y_n\rightarrow y_0}t_s(y_n)=t_s(y_0).
\end{equation}
\end{proposition}
The above proposition also holds when $t_s(y_0)=\infty$.

Before the proof, a lemma is needed here.
\begin{lemma}\label{more than one} If $Dg(y_0)\neq 0$ and
$\overline t_s(y_0)<t_s(y_0)$, then for any $t$ such that $\overline
t_s(y_0)<t<t_s(y_0)$, $(x, t)$ has more than one minimizer and  all
minimizers except $y_0$, at which $g$ obtains local minimum in
${\mathbb R^n}$, where
$x=y_0+\frac{Dg(y_0)}{|Dg(y_0)|}t$.\end{lemma}
\begin{proof}By the definition of $\overline t_s$,  we know that
$(x, t)$ has more than one minimizer. Next it will be proved by
contradiction argument, if $y_1$ is another
 minimizer for $(x, t)$ and at it,  $g$ does not attain a local
 minimum in
${\mathbb R^n}$. It implies that in any neighborhood of $y_1$, there
exists some $y_2$ such that $g(y_1)>g(y_2)$.  Then for any $t_1$
such that
 $t<t_1<t_s(y_0)$, let $x_1=y_0
+\frac{Dg(y_0)}{|Dg(y_0)|}t_1$, we have that $\overline B_{t_1}(x_1)\supset\supset
\overline B_{t}(x)$ and $\overline B_{t_1}(x_1)$ is tangent to $\overline B_{t}(x)$ at
$y_0$. Hence $y_2\in \overline B_{t_1}(x_1)$  with $g(y_2)<g(y_1)=g(y_0)$.
It implies that $y_0$ is not a minimizer for $(x_1, t_1)$. Then we
have $t_1>t_s(y_0)$. It is a contradiction with assumption that
$t_1<t_s(y_0)$.
\end{proof}
Next we prove Proposition \ref{limsup}.
\begin{proof}First we prove that
$\limsup_{y_n\rightarrow y_0}t_s(y_n)\le t_s(y_0)$.

By contradiction method, we assume $\limsup_{y_n\rightarrow
y_0}t_s(y_n)> t_s(y_0)$. Let $ \limsup_{y_n\rightarrow
y_0}t_s(y_n)=t_1$. Then there exists subsequence $\{y_j\}$ such that
\\ $ \lim_{y_j\rightarrow y_0}t_s(y_j) =t_1$. Since $y_j\rightarrow
y_0$, $Dg(y_j)\rightarrow Dg(y_0)$. Choose some $t_2$ such that
$t_s(y_0)<t_2<t_1$. Let $x_j=y_j+\frac{Dg(y_j)}{|Dg(y_j)|}t_2$. Let
$x_2=y_0+\frac{Dg(y_0)}{|Dg(y_0)|}t_2$. We have that $(x_j,
t_2)\rightarrow (x_2, t_2)$ as $j\rightarrow \infty$. Since
$t_s(y_j)>t_2$ when $j$ is large, $(x_j, t_2)$ has $y_j$ as its  minimizer. Then
$$u(x_j, t_2)=g(y_j).$$ Since $u(x,t)$ is continuous, let
$j\rightarrow\infty$, we obtain that $$u(x_2, t_2)=g(y_0).$$ It
contradicts with that $y_0$ is not a minimizer for $(x_2, t_2)$.
Then we have that
\begin{equation}\label{lims}\limsup_{y_n\rightarrow y_0}t_s(y_n)\le
t_s(y_0).\end{equation}

Next we prove equality (\ref{=}). 
For any $t< t_s(y_0)$, let $x=y_0+\frac{Dg(y_0)}{|Dg(y_0)|}t$.

 By Lemma \ref{more than one},  $(x, t)$ has unique minimizer $y_0$ with $Dg(y_0)\neq 0$. In other words, if $(x,t)$ has other minimizers, then $g$ attains   local minimum at these minimizers.


For any $y_j\in H^-(y_0)$ (see Definition \ref{H}) such that $y_j\rightarrow y_0$, $(x_j, t)\rightarrow (x, t)$, where $x_j=y_j+\frac{Dg(y_j)}{|Dg(y_j)|}t$. 
Then for any $\bar y_j\in L(x_j, t)$, 
 we have that $g(\bar y_j)\le g(y_j)<g(y_0)$. From above, we have known
  that except $y_0$, at other minimizers for $(x,t)$ if there is any,
  $g$ attains local
 minimum in ${\mathbb R^n}$ and  by $L(\cdot, \cdot)$ is upper
 semi-continuous multi-function,
  we know that $\bar y_j$ can only be around $ y_0$. Since $g(\bar y_j)<g(y_0)$,
  it implies that  $\bar y_j\in H^-(y_0)$ and $Dg(\bar y_j)\neq 0$.

  Next we claim that for any $y\in L(x_j, t)$,   $Dg(y)\neq 0$.

Proof of the claim: Recall that  $\bar y_j$ is one minimizer for
$(x_j, t)$. For any $y\in \overline B_t(x_j)$, if $y\in
\overline B_{t_s(y_0)}(x_s(y_0))$, then $g(y)\ge g(y_0)>g(\bar y_j)$, it
implies that $y$ is not a minimizer for $(x_j, t)$. If $y$ is not in
$\overline B_{t_s(y_0)}(x_s(y_0))$, when $\bar y_j$ is sufficiently close to
$y_0$, then $y$ is  close to $\bar y_j$. It implies that $Dg(y)\neq
0$. The proof of the claim is complete.

    By the claim above, we obtain that $t_s(\bar y_j)=
    \overline t_s(\bar y_j)$.   
 Otherwise, if $t_s(\bar y_j)>\overline t_s(\bar y_j)$, then there must exist
 a minimizer $y_1$ such that $g$ attains a local minimum in
 ${\mathbb R^n}$ at $y_1$, which means $Dg(y_1)=0$. It contradicts with the claim.
   Furthermore, since $\bar y_j\in L(x_j, t)$, we have that
   $\overline t_s(\bar y_j)=
     t_s(\bar y_j)\ge t$. Hence $$\limsup_{\bar y_j\rightarrow y_0}\overline t_s(\bar y_j)\ge t.$$
 It implies that $$\limsup_{y_n\rightarrow y_0}\overline t_s(y_n)\ge t.$$
Since $t$ is arbitrary as long as $t<t_s(y_0)$, we have that
$$\limsup_{y_n\rightarrow y_0}\overline t_s(y_n)\ge t_s(y_0).$$ With
(\ref{lims}), the following equalities hold.
$$\limsup_{y_n\rightarrow y_0}\overline
t_s(y_n)=\limsup_{y_n\rightarrow y_0} t_s(y_n) = t_s(y_0).$$
 \end{proof}
 From the proof of the proposition above, we can  see that
 $t_s(y_0)$
  can be obtained by a limit of $\overline t_s(\bar y_j)$ for some sequence
  $\{\bar y_j\}$ such that $\bar y_j\in H^-(y_0)$ and $\bar y_j\rightarrow
y_0$. Hence the more accurate version can be written as follows:
 \begin{corollary}\label{}For any $y_0$ with $Dg(y_0)\neq 0, $\begin{equation}
\limsup_{y_n\in H^-(y_0),\,y_n\rightarrow y_0}\overline t_s(y_n)=
\limsup_{y_n\in H^-(y_0),\,y_n\rightarrow y_0}t_s(y_n)=t_s(y_0).
\end{equation}
\end{corollary}
Next how about $\limsup t_s(y_n)$ when $y_n\in H^+(y_0)$? We have
the following:
\begin{proposition}\label{limsH+}$$\limsup_{y_n\in H^+(y_0),\,
y_n\rightarrow y_0}t_s(y_n)=\limsup_{y_n\in H^+(y_0),\,
y_n\rightarrow y_0}\overline t_s(y_n) =\overline t_s(y_0).$$
\end{proposition}
\begin{proof}For any $t>\overline t_s(y_0)$, any  $y_n\in H^+(y_0)$,
we have that $g(y_n)>g(y_0)$. Let
$x=y_0+\frac{Dg(y_0)}{|Dg(y_0)|}t$. By  $t>\overline t_s(y_0)$,
there exists some $y_1$ in $B_t(x)$ such that $y_1\neq y_0$ and
$g(y_1)\le g(y_0)$. Notice that by definition, $B_t(x)$ is an open
ball. When $y_n$ is sufficiently close to $y_0$, we have that
$y_1\in B_t(x_n)$, where $x_n=y_n+\frac{Dg(y_n)}{|Dg(y_n)|}t$. Hence
$g(y_1)<g(y_n)$. It implies that $ t_s(y_n)<t$. Therefore
\begin{equation}\label{limsH+>}\limsup_{y_n\in H^+(y_0),\,y_n
\rightarrow y_0}t_s(y_n) \le
\overline t_s(y_0).\end{equation}

Next, for any $t<\overline t_s(y_0)$, $y_0$ is unique minimizer for
$(x, t)$ in $\overline B_t(x)$, where $x=y_0+\frac{Dg(y_0)}{|Dg(y_0)|}t$.
Choose a sequence $\{y_k\}\subset B_t(x)$ such that $y_k\rightarrow
y_0$ and $\overline B_t(x_k)\subset B_{\overline t_s(y_0)}(\overline
x_s(y_0))$, where $x_k=y_k+\frac{Dg(y_k)}{|Dg(y_k)|}t$. Since $y_0$
is unique minimizer for $(x, t)$ in $\overline B_t(x)$ and $y_k\in B_t(x)$,
we have that $g(y_k)>g(y_0)$, it implies that $y_k\in H^+(y_0)$.

Case 1, if $\overline t_s(y_k)\ge t$ for all $k$, then $\limsup_{y_n\in
H^+(y_0),\,y_n\rightarrow y_0}\overline t_s(y_n)\ge t$.

Case 2, if there exists a subsequence of $\{y_k\}$, (still denoted as $\{y_k\}$) such that $\overline t_s(y_k)< t$, then let $\bar y_k\in
L(x_k, t)$,  there must exist some
 subsequence $\{\bar y_{k_j}\}$ and some $\bar y$ such that
  $\bar y_{k_j}\rightarrow \bar y$ and
 $\bar y\in L(x, t)$. It implies that $\bar y=y_0$ due to that
$y_0$ is unique minimizer for $(x, t)$. That means that $\bar
y_{k_j}\rightarrow  y_0$. Since $\bar y_{k_j}\in L(x_{k_j}, t)$, we
have that $t_s(\bar y_{k_j})\ge t$. Next we want to prove $\overline t_s(\bar
y_{k_j}) = t_s(\bar y_{k_j})$ by contradiction methods.  Assume that $\overline t_s(\bar
y_{k_j}) < t_s(\bar y_{k_j})$, then there exists some $\tilde
y_{k_j}\in L(y_{k_j}, t)$ such that $Dg(\tilde y_{k_j})=0$. Then
there exist some subsequence of $\{\tilde y_{k_j}\}$, still denoted
as $\{\tilde y_{k_j}\}$ and some $\tilde y$ such that $\tilde
y_{k_j}\rightarrow \tilde y$ with $\tilde y\in L(x, t)$ and
$Dg(\tilde y)=0$. Since $Dg(y_0)\neq 0$, we have that $y_0\neq
\tilde y$. It contradicts with that $y_0$ is the unique minimizer in
$L(x, t)$. Hence the assumption is incorrect.   $\overline t_s(\bar
y_{k_j}) =t_s(\bar y_{k_j})$ must hold. By  $t_s(\bar y_{k_j})\ge
t$ for  any $\bar y_{k_j}$, we immediately have $\overline t_s(\bar y_{k_j})\ge t$ for  any $\bar y_{k_j}$. Since $\bar
y_{k_j}\in \overline B_t(x_{k_j})\subset B_{\overline t_s(y_0)}(\overline
x_s(y_0))$, we have that $g(\bar y_{k_j})>g(y_0)$, it implies that
$\bar y_{k_j}\in H^+(y_0)$. Therefore $$\limsup_{y_n\in
H^+(y_0),\,y_n\rightarrow y_0}\overline t_s(y_n)\ge t.$$ Due to that
$t$ is arbitrary as long as $t<\overline t_s(y_0)$, we have that
$$\limsup_{y_n\in H^+(y_0),\,y_n\rightarrow y_0}\overline
t_s(y_n)\ge \overline t_s(y_0).$$ By (\ref{limsH+>}), the following
holds:$$\limsup_{y_n\in H^+(y_0),\, y_n\rightarrow
y_0}t_s(y_n)=\limsup_{y_n\in H^+(y_0),\, y_n\rightarrow
y_0}\overline t_s(y_n) =\overline t_s(y_0).$$
\end{proof}
From Proposition \ref{limsH+}, we can deduce the property of
$\liminf_{y_n\rightarrow y_0}t_s(y_n)$ immediately.
\begin{proposition}\label{liminf} $\liminf_{y_n\rightarrow y_0}t_s(y_n)\le
\overline t_s(y_0)$.
\end{proposition}
It is trivial to see the following
\begin{equation}\label{limi}\liminf_{y_n\rightarrow y_0}\overline t_s(y_n)\le
\overline t_s(y_0).
\end{equation}

If the equality in Proposition \ref{liminf} can be obtained, then
$t_s$ is similar to a continuous function.  We can say more about
the property of the set of singularity points like in \cite{ZTW}.
Actually, for $C^2$ initial data, the equality can be obtained,  which
will be seen  in the next section.

But for $C^1$ initial data, it is possible that the inequality can be obtained.  
Recall we mentioned in the introduction,   for strictly convex, superlinear and smooth Hamiltonians in one dimensional case, for any
initial data in a second
category in $C^1$ class, 
the nondifferentiable points of the solution $u$
are dense in ${\mathbb R}\times \{t \,| t>0\}$ (\cite{Li2}). 
  So it is very likely that the  inequality could also be
strict in Proposition \ref{liminf} or (\ref{limi}) for the Eikonal equation.

Then what will
happen when  $\liminf_{y_n\rightarrow y_0}t_s(y_n)<
\overline t_s(y_0)$? 
We find out the following  proposition, it tell us that
 when the inequality is strict, for
each point on the characteristic segment from $y_0$ in the time
interval ($\liminf_{y_n\rightarrow y_0}t_s(y_n)$, $\overline
t_s(y_0)$), it is a cluster point of nondifferentiable points.

 \begin{proposition}\label{dense}If $\liminf_{y_n\rightarrow y_0}\overline t_s(y_n)
 <\overline t_s(y_0)$, then for any $t_0$ such that $\liminf_{y_n\rightarrow y_0}\overline t_s(y_n)<t_0< \overline t_s(y_0)$, let $x_0=y_0+\frac{Dg(y_0)}{|Dg(y_0)|}t_0$, for any $\epsilon$, there exists some $x_{\epsilon}\in \overline B_{\epsilon}(x_0)$ such that $(x_{\epsilon}, t_0)$ is a nondifferentiable point, where $\overline B_{\epsilon}(x_0)$ is the closed ball centered in $x_0$ with radius $\epsilon$.
 \end{proposition}
 \begin{proof}We will prove by contradiction argument. Assume there
 exists some $\epsilon$ small enough such that for any
 $x\in \overline B_{\epsilon}(x_0)$, $(x, t_0)$ is a differentiable point. Due
 to the definition of $\overline t_s(y_0)$ and the upper semi-continuity
 of $L$, for any $y\in L(x, t_0)$, we have that $y\approx y_0$, so
 $Dg(y)\neq 0$. Since $(x, t_0)$ is a differentiable point, by
 Theorem \ref{diff}, we know that $(x, t_0)$ has unique minimizer.
 Then $\overline B_{\epsilon}(x_0)$ and $L(\overline B_{\epsilon}(x_0), t_0)$ are one-
 to- one correspondence, where $L(\overline B_{\epsilon}(x_0), t_0)$ denotes
 the set of minimizers for points $(x, t_0)$ when $x\in
 \overline B_{\epsilon}(x_0)$.

 For any $y\in L(\overline B_{\epsilon}(x_0), t_0)$, let
 $x=y+\frac{Dg(y)}{|Dg(y)|}t_0$.
 Let $x_n=y_n+\frac{Dg(y_n)}{|Dg(y_n)|}t_0$,
 where $y_n\in L(\overline B_{\epsilon}(x_0), t_0)$. When $y_n\rightarrow y$,
 we have that $x_n\rightarrow x$. It follows that
 $L^{-1}$ is continuous on $L(\overline B_{\epsilon}(x_0), t_0)$. Conversely,
 for any $x\in \overline B_{\epsilon}(x_0) $, when $x_n\rightarrow x$,
 by assumption that $(x, t_0)$ has unique minimizer $y\in L(x, t_0)$
 and the upper semi-continuity of $L$, we have that $y_n\rightarrow y$, where $y_n$ is unique minimizer in $L(x_n, t_0)$. It follows that $L$ is continuous on $(\overline B_{\epsilon}(x_0), t_0)$.

 Then $L(\overline B_{\epsilon}(x_0), t_0)$ is homeomorphism to the disk $(\overline B_{\epsilon}(x_0), t_0)$. Therefore $L(\overline B_{\epsilon}(x_0), t_0)$ is a closed simply connected set. Since $y_0$ is an interior point of $ L(\overline B_{\epsilon}(x_0), t_0)$, for any $\{y_n\}$ such that $y_n\rightarrow y_0$, we have that $y_n\in L(\overline B_{\epsilon}(x_0), t_0)$. It follows that $\liminf_{y_n\rightarrow y_0}\overline t_s(y_n)\ge t_0$. It contradicts with  $\liminf_{y_n\rightarrow y_0}\overline t_s(y_n)$ $<$ $t_0$.
 \end{proof}
  The above proposition holds when $\overline t_s(y_0)=\infty$.

\subsection{Global structure of $\Sigma\cup T_1$, $\Sigma$ and $\overline \Sigma$ and regularity of solutions }
Next let us collect    termination points of the  characteristic
from $y_0$ for all $y_0$ with $Dg(y_0)\neq 0$ and points in between the first
termination point $(\overline x_s(y_0), \overline t_s(y_0))$ and the
second termination point $( x_s(y_0), t_s(y_0))$ when $\overline
t_s(y_0)<t_s(y_0)$ as follows:
 \begin{definition}\label{M} Let $M=\{y_0\in {\mathbb R^n}|Dg(y_0)\neq 0, \overline t_s(y_0)<\infty\}$.
Let $T(M)=\{(x,t)|y_0\in M, x=y_0+\frac{Dg(y_0)}{|Dg(y_0)|}t, (1)\,
t\in [\overline t_s(y_0),  t_s(y_0)] \;\mbox{if} \;
t_s(y_0)<\infty;\quad (2)\, t\in [\overline t_s(y_0),
 \infty)
   \;\mbox{if}\; t_s(y_0)=\infty \}$. We call $T(M)$  the set of termination points of characteristics from points with
 $Dg\neq 0$.\end{definition}
\begin{definition}
 We denote the set of all
 nondifferentiable points as $\Sigma$. Denote 
 the set of termination points in $T(M)$ such as those points have unique minimizer as $T_1$.
 \end{definition}
   Any point in $\Sigma$ has more than one minimizer. Any point in $T_1$ is a differentiable point.

 \begin{remark}   Points in $T_1$ are generating points of nondifferentiable points. The reason is that recall Proposition \ref{cluster} that any point in $T_1$ is a cluster point of nondifferentiable points at later time and in section 5, we will show that for $C^2$ initial data, the norm of  second derivatives of $u$ is infinite at any  point in $T_1$.  In section 5, it will  also be shown that $T_1$ is a generalization of the concept of the set of conjugate points  $\Gamma$ (which can only be defined for $C^2$ initial data) for  function classes weaker than the $C^2$ class,   since when the initial data is of $C^2$ class,  we have that $T_1=\Gamma\setminus \Sigma$.  \end{remark}
 It is easy to see that
  \begin{proposition}$T(M)=\Sigma\cup T_1$.
  \end{proposition}\begin{proof}For any point $(x_0, t_0)$ in $\Sigma$, by Theorem
 \ref{diff}, it has more than one minimizer and one of minimizers
 $y_0$ has $Dg(y_0)\neq 0$. It implies that $(x_0, t_0)$ is on the
  effective
 characteristic from $y_0$ with $\overline t_s(y_0)\le t_0\le
 t_s(y_0)$. Hence $(x_0, t_0)\in T(M)$. It implies that $\Sigma\subset T(M)$. For any point in $T(M)\setminus \Sigma$, it is a differentiable termination point of a characteristic from some $y_0$ with $Dg(y_0)\neq 0$. Hence it is a point with unique minimizer in $T(M)$ by Theorem
 \ref{diff}. It follows that $ T(M)\setminus \Sigma=T_1$. 
 \end{proof}
 Next since by Theorem
 \ref{diff}, we know that any point in $T_1$ is a cluster point of nondifferentiable points,  so the following holds:
 \begin{proposition}\label{over}$\overline {T(M)}=\overline \Sigma$.
 \end{proposition} 

  But the question is whether $T_1$ is an empty set or not.  Let us see, when $n=1$, $T_1=\emptyset$ holds, because for any $y_0$ with $\overline t_s(y_0)<\infty$, it is easy
   to prove that $(\overline x_s(y_0), \overline t_s(y_0))$ must have
   more than one minimizer. The reason is that there are only two
   direction $-1$ and $1$ when $Dg(y_0)\neq 0$.
    See more detail in Lemma \ref{n=1}. While when
    $n>1$,  $T_1$ may not be $\emptyset$. That is,
    $(\overline x_s(y_0), \overline t_s(y_0))$ may have unique
    minimizer. For instance, for some  initial data $g$ and
     some point $(x_0, t_0)$, where
    $x_0=y_0+\frac{Dg(y_0)}{|Dg(y_0)|}t_0$,  the curvature of  the  two
    surfaces $\partial B_{t_0}(x_0)$ and $H(y_0)$  at $y_0$ are same,
     and $H(y_0)$  is in the exterior of $\partial B_{t_0}(x_0)$, and
     $y_0$ is unique minimizer for $(x_0, t_0)$ (Recall $H(y)$ is
     determined by initial data $g$, so we can choose some $g$ satisfying
     conditions above).  Then for any $t>t_0$,
     we have that the curvature of  $\partial B_{t}(x)$ at $y_0$ is less
      than the curvature of $H(y_0)$ at $y_0$, where
      $x=y_0+\frac{Dg(y_0)}{|Dg(y_0)|}t$. The picture is that $H(y_0)$ is
      inside two surfaces $\partial B_{t_0}(x_0)$ and
      $\partial B_{t}(x)$. Hence locally, $H(y_0)$ is inside  the ball
      $\overline B_t(x)$. Since $H(y_0)$ separates the neighborhood of $y_0$ into two
       parts $H^+(y_0)$ and $H^-(y_0)$,
       it implies that there exists
      some $y_1\in \overline B_t(x)$ such that $y_1\in H^-(y_0) $, i.e.,
      $g(y_1)<g(y_0)$. Hence $(x, t)$
      does not have $y_0$ as its minimizer. Therefore
      $t_0=\overline t_s(y_0)= t_s(y_0)$. That is,
      $(\overline x_s(y_0), \overline t_s(y_0))$ coincides with
      $( x_s(y_0), t_s(y_0))$, and it  has unique minimizer.
      Therefore it is in $T_1$.

  For points in $T_1$,
  the following lemma holds:
  \begin{lemma}For any $(x_0, t_0)\in T_1$, let $y_0$ be the unique minimizer for $(x_0, t_0)$, we have that $\overline t_s(y_0)= t_s(y_0)$.
  \end{lemma}
  \begin{proof}It will be proved by contradiction methods. For any point $(x_0, t_0)$
   in $T_1$, since $T_1$ is a subset of $T(M)$, by Proposition
   \ref{cluster}, if $\overline t_s(y_0)< t_s(y_0)$, then
   $(\overline x_s(y_0), \overline t_s(y_0))$, $( x_s(y_0),  t_s(y_0))$
   and any point in between are all nondifferentiable points. Hence any
    of these points have more than one minimizer. It implies that
    $(x_0, t_0)$ has more than one minimizer,
     which is a contradiction. Therefore $\overline t_s(y_0)= t_s(y_0)$.
  \end{proof}



No matter for $C^1$ or $C^2$ initial data for the Eikonal equation, in general, $T(M)=\Sigma\cup T_1$ is a only proper subset of $\overline \Sigma$.  $\Sigma\cup T_1$
contains maybe only part of cluster points of $\Sigma$, some cluster points of $\Sigma$ could also be
the termination points of characteristics from $y_0$ with
$Dg(y_0)=0$.




  Next the following set is defined:
 \begin{equation}
 H_{y_0}=\{y\in {\mathbb R^n}|(y-y_0)\cdot Dg(y_0)>0\}.
 \end{equation}
 It is easy to see that $H_{y_0}$ is equivalent to the following
 set, (thus denote by $H_{y_0}$ again)\begin{equation}H_{y_0}=\{y\in {\mathbb R^n}|y\in B_t(x),
 \mbox{where}\,x=y_0+\frac{Dg(y_0)}{|Dg(y_0)|}t, \, \mbox{for some } \,
 t>0\}.
\end{equation}
It means that $H_{y_0}$ covers all the balls with center
$y_0+\frac{Dg(y_0)}{|Dg(y_0)|}t$ and radius $t>0$. With this picture,
the proof of next proposition looks more transparent.

 Another set is defined as follows:
 \begin{equation}
 D=\{y_0\in {\mathbb R^n}| Dg(y_0)\neq 0,  \exists\; \bar y \in H_{y_0}\;\mbox{such that}\; g(\bar y)\le g(y_0)\}.
 \end{equation}
  We find out that $M$ and $D$ are equivalent, which is  the following proposition:
 \begin{proposition}\label{M=D}$M=D$. i.e., $\{y_0\in {\mathbb R^n}|Dg(y_0)\neq 0, \overline t_s(y_0)<\infty\}=\{y_0\in {\mathbb R^n}| Dg(y_0)\neq 0,  \exists\; \bar y \in H_{y_0}\;\mbox{such that}\; g(\bar y)\le g(y_0)\}$.
 \end{proposition}
 \begin{proof} For any $y_0\in M$, there are two cases. 1) $t_s(y_0)=\overline t_s(y_0)<\infty$, then there exists some $\bar y\in H_{y_0}$ such that $g(\bar y)<g(y_0)$.

 2) $\overline t_s(y_0)<t_s(y_0)$, then there exists some $\bar y\in H_{y_0}$ such that $g(\bar y)=g(y_0)$.

 Both cases imply that $y_0\in  \{y_0\in{\mathbb R^n}| Dg(y_0)\neq 0,  \exists\; \bar y \in H_{y_0},  g(\bar y)\le g(y_0)\}$.

 Conversely, for any $y_0\in  \{y_0\in{\mathbb R^n}| Dg(y_0)\neq 0,  \exists\; \bar y \in H_{y_0},  g(\bar y)\le g(y_0)\}$, 1) if $g(\bar y)< g(y_0)$, then $t_s(y_0)<\infty$, it implies that  $\overline t_s(y_0)<\infty$.

 2) If $g(\bar y)= g(y_0)$, then $\overline t_s(y_0)<\infty$.
 \end{proof}
 Let us define $E$ as follows:
 \begin{equation}E= \{y_0\in{\mathbb R^n}| Dg(y_0)\neq 0,  \exists\;
 \bar y \in H_{y_0},\,\mbox{such that }\;  g(\bar y)< g(y_0)\}
 \end{equation}
 \begin{lemma}\label{E} $E$ is an open set.
 \end{lemma}
 \begin{proof}
For any $y_0\in E$, we have that $Dg(y_0)\neq 0$,  and $\exists
 \bar y \in H_{y_0}$ such that $ g(\bar y)< g(y_0)$.  There exists some
 open set $U$ such that
$y_0\in U$,
 and for any $y\in U$, $Dg(y)\neq 0$ and $g(y)>g(\bar y)$, due
 to that $Dg$ is continuous.

Next we prove that $\bar y\in H_y$. That is, $(\bar y-y)\cdot Dg(y)>0$.
Let $a=dist(\bar y, \partial H_{y_0})$, i.e. $a=(\bar y-y_0)\cdot \frac{Dg(y_0)}{|Dg(y_0)|}$. For any $\epsilon>0$, there exists some $\delta$, such that when $|y-y_0|<\delta$, we have that $|Dg(y)-Dg(y_0)|<\epsilon$.
\begin{eqnarray*}
(\bar y-y)\cdot Dg(y)=&&(\bar y-y_0)\cdot Dg(y)+( y_0-y)\cdot Dg(y)\\
=&&(\bar y-y_0)\cdot Dg(y_0)+(\bar y-y_0)\cdot  (Dg(y)-Dg(y_0))\\&&+( y_0-y)\cdot Dg(y_0)+( y_0-y)\cdot (Dg(y)-Dg(y_0))\\
\ge&&a|Dg(y_0)|-\epsilon |\bar y-y_0|-\delta |Dg(y_0)|-\epsilon \delta>0
\end{eqnarray*}
when $\epsilon$, $\delta$ are sufficiently small. It implies that $E$ is open.
 \end{proof}
 Since $E$ is open, therefore $E$ is composed of at most countable path connected components. i.e.,
 \begin{equation}\label{E_i}E=\bigcup E_i,\end{equation} where $E_i$ is a path connected component of $E$.

 Let us define set $J$ as follows:
 \begin{equation}J= \{y_0\in{\mathbb R^n}| Dg(y_0)\neq 0,  t_s(y_0)<\infty\}
 \end{equation}
 We have the equivalence of $E$ and $J$ as follows:
 \begin{proposition}\label{E=J}$E=J$; i.e. $\{y_0\in{\mathbb R^n}| Dg(y_0)\neq 0,  \exists\; \bar y \in H_{y_0},  g(\bar y)< g(y_0)\}=\{y_0\in{\mathbb R^n}| Dg(y_0)\neq 0,  t_s(y_0)<\infty\}$.
 \end{proposition}
 The proof is trivial.
 Next we have the construction of some  path connected sets $F_i$.
 \begin{proposition}\label{F_i}$M=D= \bigcup F_i$, where $F_i\supset E_i$.  $F_i$ is also path connected.
 \end{proposition}

\begin{proof}
For any $y_0\in D\setminus E$, we have that $Dg(y_0)\neq 0$ and
$\overline t_s(y_0)<t_s(y_0)=\infty$. By Proposition \ref{limsH+}
that $\limsup_{y_n\in H^+(y_0),\,y_n\rightarrow y_0}t_s(y_n) =
\overline t_s(y_0)$, then there exists a small neighborhood  $U$ of
$y_0$, such that for any $y\in U\cap H^+(y_0)$, we have that
$t_s(y)<\infty$.
It implies that  $U\cap H^+(y_0)\subset E$. Furthermore, $U\cap
H^+(y_0)$ is path connected, then we have that $U\cap
H^+(y_0)\subset E_i$ for some $i$. Since $y_0$ is  on the boundary
of $U\cap H^+(y_0)$,  for any point in $U\cap H^+(y_0)$, there will
be a continuous curve connecting it and $y_0$.  Therefore there is a
path connected  set $F_i\supset E_i$, such that $y_0\in F_i$ and
$F_i\subset D$.
\end{proof}
\begin{lemma}$F_i\cap F_j=\emptyset$, $i\neq j$.
\end{lemma}
\begin{proof}Since $E_i$ is a path connected component,
$E_i\cap E_j=\emptyset$ for $i\neq j$. By the argument in the proof
of Proposition \ref{F_i}, hence $(F_j\setminus E_j)\subset
(D\setminus E)$, therefore $E_i\cap F_j=\emptyset$. Next we will
prove by a contradiction argument. If we assume $F_i\cap
F_j\neq\emptyset$, then $\exists y_0\in F_i\cap F_j$. Since $E_i\cap
E_j=\emptyset$, $E_i\cap F_j=\emptyset$ and $F_i\cap E_j=\emptyset$,
hence $y_0\in (F_i\setminus E_i)\cap(F_j\setminus E_j)$. By the
argument in the proof of Proposition \ref{F_i} again, we have that
$U\cap H^+(y_0)\subset E_i$ and $U\cap H^+(y_0)\subset E_j$. It
implies that $U\cap H^+(y_0)\subset E_i\cap E_j=\emptyset$.  It is a
contradiction. Then the proof is complete.
\end{proof}
From propositions above, it can be seen that for any $y_0\in E_i$,
we have that $\overline t_s(y_0)\le t_s(y_0)<\infty$. For any
$y_0\in F_i\setminus E_i$, we have that $\overline t_s(y_0)<
t_s(y_0)=\infty$.

$F_i$ is path connected, but it is not necessarily a path connected component of $M$.  So there  exists a  path connected component of $M$, we call it $R_j$, such that $R_j\supset F_i$. Then we have
$M= \bigcup R_j$, where   $R_j$ is a path connected component of $M$.

 The number of path connected components of $M$ may be less than
the number of path connected components of $E$.  Since for some $i$
and $j$, $E_i$ and $E_j$ are two path connected components of $E$,
there may exist a path connected component $R_k$ of $M$ such that
$R_k\supset E_i\cup E_j$.

\begin{proposition}\label{T(Ei)}
 $T(E_i)$ is 1) either a connected set,
 2) or if  $T(E_i)$ is not a connected set,  then for any two
 disjoint closed subset $A$ and $B$ of  $T(E_i)$ such that
 $T(E_i)=A\cup B$, we have   $distance(A, B)=0$.
 \end{proposition}
\begin{proof}$T(E_i)$ is the set of all termination points of
characteristics
 from $y_0\in E_i$ and points between $(\overline x_s(y_0),
 \overline t_s(y_0))$ and $(x_s(y_0), t_s(y_0))$.  If $T(E_i)$ is not a connected
set, then for any two disjoint closed subsets $A$ and $B$ of
$T(E_i)$ such that $T(E_i)=A\cup B$,  we have that $(L(A)\cap
L(B))\cap E_i=\emptyset $ and $(L(A)\cup L(B))\cap E_i=E_i$,  due to
that
 for any $y_0\in E_i$, $Dg(y_0)\neq 0$, hence there is only one characteristic
 starting from $y_0$. Recall (\ref{L(x, t)}), $L(A)$ here means  the set of minimizers for all points $(x, t)$ in $A$.  Since $E_i$ is connected,
$L(A)\cap E_i$ and $ L(B)\cap E_i$ can not both be closed sets of
$E_i$.

We assume that $L(A)\cap E_i$ is not a closed set of $E_i$. Then there
exists some sequence $\{y_n\}\in L(A)\cap E_i$ and some
$y_0\in L(B)\cap E_i$, such that $y_n\rightarrow y_0$. Let
$t_1=\limsup_{y_n\in L(A)\cap E_i, y_n\rightarrow y_0}t_s(y_n)$.

1.First if we assume   $t_1\ge\overline t_s(y_0)$,  since
$t_s(y_0)<\infty$, then
  there must exist some sequence $\{y_j\}\subset L(A)\cap E_i$ such that
 $y_j\rightarrow y_0$ and $( x_s(y_j),  t_s(y_j))\rightarrow (x , t
 )$, where $x=y_0+\frac{Dg(y_0)}{|Dg(y_0)|}t$ and $t\in [\overline t_s(y_0), t_s(y_0)]$, due to that $\limsup_{y_n\rightarrow y_0}t_s(y_n)=t_s(y_0)$ in Proposition \ref{limsup}.
 Since $(x_s(y_j),  t_s(y_j))\in A$, and
 $(x, t)\in B$, it contradicts with
 that $A$ and $B$ are  disjoint closed subsets  of
 $T(E_i)$. Hence we must have $t_1<\overline t_s(y_0)$.

2.   Then when $t_1<\overline t_s(y_0)$,  for any $t$ such that
$t_1<t<\overline t_s(y_0)$, let $x=y_0+\frac{Dg(y_0)}{|Dg(y_0)|}t$,
by Proposition \ref{dense}, we have that for any $\epsilon$, there
exists some $x_\epsilon\in \overline B_\epsilon(x)$ such that $(x_\epsilon,t)$
is a nondifferentiable point.

 Next we need to prove $(x_\epsilon,t)\in B$. Since $t<\overline t_s(y_0)$, $y_0$ is unique element in $L(x,t)$.
 Due to that  $L$ is upper
 semi-continuous, when $\epsilon$ is sufficiently small, for any
 element $y_\epsilon$ in $L(x_\epsilon, t)$,
 we have that $y_\epsilon\approx y_0$. 
 Since $E_i$ is an open set,  $y_\epsilon\in E_i$.

Due to that $t>t_1$, hence
 $y_{\epsilon}\in L(B)$. It follows that
 $(x_{\epsilon}, t)$
 is in $B$. By the arbitrarity of $t$  as long as $t>t_1$,
 $(x_1, t_1)$ is a cluster point of points in $B$,
 where $x_1=y_0+\frac{Dg(y_0)}{|Dg(y_0)|}t_1$. At the same time,
 $(x_1, t_1)$ is also a cluster point of points in $A$. So we have that
 $distance(A, B)=0$.

\end{proof}Since $2)$ in Proposition \ref{T(Ei)} is a natural
phenomena coming from the proof, we give a definition for it.
\begin{definition}\label{almost} We  define that a set $C$ is almost connected if for any two
 disjoint closed subset $A$ and $B$ of  $C$ such that
 $C=A\cup B$, we have $distance(A, B)=0$.
\end{definition}
Next we will see what happens to $T(F_i)$.
\begin{proposition}\label{T(F_i)}
 If $T(E_i)$ is a connected set, then $T(F_i)$ is also a connected set.
 If $T(E_i)$ is an almost connected set, then $T(F_i)$ is  also an almost connected set.
 \end{proposition}
 \begin{proof}For any $(x_0, t_0)\in T(F_i\setminus E_i)$,
 $\exists y_0\in L(x_0, t_0)$ such that  $Dg(y_0)\neq 0$. Therefore
 $y_0\in F_i\setminus E_i$. First by the result in the proof of
 Proposition \ref{F_i}, we know that there is some small neighborhood
   $U$ of $y_0$ such that $U\cap H^+(y_0)\subset E_i$. Then by Proposition
  \ref{limsH+} that $\limsup_{y_n\in H^+(y_0),\,
y_n\rightarrow y_0}t_s(y_n)=\overline t_s(y_0)$,  there exists some
sequence $\{y_k\}\subset H^+(y_0)\subset E_i$ such that
$y_k\rightarrow y_0$ and $\lim_{n\rightarrow \infty}t_s(y_k)=
\overline t_s(y_0)$. Therefore $(x_s(y_k), t_s(y_k))\in T(E_i)$ goes to
$(\overline x_s(y_0), \overline t_s(y_0))\in T(F_i\setminus E_i)$.
$(x_0, t_0)$ is connected with $(\overline x_s(y_0), \overline t_s(y_0))$
 by the characteristic  segment which is a subset of $T(F_i\setminus E_i)$. 
 Hence we obtain the desirable result.
 \end{proof}
 $T(F_i)$ may intersect with $T(F_j)$, for some $i\neq j$. Hence the
number of connected components or almost connected components of
$T(M)$ may be less than the number of $i$ of $M=\bigcup F_i$.

It is obvious that we have the following proposition:
\begin{proposition}If for some $(x, t)$, $L(x, t)=
\{y_0, y_1, y_2, ..., y_n\}$, $y_i\in F_i$ and $T(F_i)$ is a
connected set, for $i=0, 1, ..., n$, Then $\bigcup_{i=0}^{i=n}
T(F_i)$ is a connected set.
\end{proposition}
The above property is same for almost connected sets.

 Combine Proposition \ref{F_i}, $\Sigma\cup T_1=T(M)=\bigcup T(F_i)$, Proposition
 \ref{T(Ei)} and Proposition \ref{T(F_i)},
the following theorem can be obtained:
\begin{theorem}\label{thm1} $\Sigma\cup T_1$ is composed of at most
 countable connected components or almost connected components. That is,
  $\Sigma\cup T_1=\bigcup S_j$,  each  $S_j$ is 1) either a
connected component,
 2) or an almost connected component. Each $S_j$ is a finite or
 countable union of some $T(F_i)$.

\end{theorem}

Next we study $\Sigma$, the set of all nondifferentiable points of
$u$. It is easy to see that the following lemma holds:
\begin{lemma}$\Sigma=T(M)\setminus \{(x, t)|(x, t)\,\mbox{has unique
minimizer}\}$.
\end{lemma}
Therefore \begin{eqnarray*}\Sigma&=&(\cup T(F_i))\setminus \{(x,
t)|(x, t)\,\mbox{has unique minimizer}\}\\&=&\bigcup (
T(F_i)\setminus \{(x, t)|(x, t)\,\mbox{has unique
minimizer}\}).\end{eqnarray*} We study the property of
$T(F_i)\setminus \{(x, t)|(x, t)\,\mbox{has unique minimizer}\}$
first. The following lemma holds:
\begin{lemma}$T(F_i)\setminus \{(x, t)|(x, t)\,\mbox{has unique minimizer}\}$
 is a connected set or an almost connected set.
\end{lemma}
\begin{proof}1) If $T(F_i)\setminus \{(x, t)|(x, t)\,\mbox{has unique
minimizer}\}$ is connected, then proposition holds.

 2) If $T(F_i)\setminus \{(x, t)|(x, t)\,\mbox{has unique
minimizer}\}$ is not connected, then for any two disjoint closed
subsets $A$ and $B$ of $T(F_i)\setminus \{(x, t)|(x, t)\,\mbox{has
unique minimizer}\}$  such that $A\cup B=T(F_i)\setminus \{(x,
t)|(x, t)\,\mbox{has unique minimizer}\}$, we need to prove that
$distance (A, B)=0$.

It will be proved by contradiction method. Assume that $distance (A,
B)\neq 0$. First we see the set $T(F_i)\cap \{(x, t)|(x,
t)\,\mbox{has unique minimizer}\}$. For any $(x, t)\in T(F_i)\cap
\{(x, t)|(x, t)\,\mbox{has unique minimizer}\}$, $(x, t)$ must be in
$T(E_i)$. It implies that the unique minimizer $y$ for $(x, t)$ is
in $E_i$.  By Proposition \ref{cluster}, there exist $(x_n, t_n)\in
\Sigma$ such that $(x, t)$ is a cluster point of $\{(x_n, t_n)\}$.
Let $y_n\in L(x_n, t_n)$, we have that $y_n\rightarrow y$. Since
$E_i$ is open,  $y_n\in E_i$. Therefore $(x_n, t_n)\in T(E_i)$,
furthermore, $(x_n, t_n)\in T(F_i)\setminus \{(x, t)|(x,
t)\,\mbox{has unique minimizer}\}$. That is,  $(x_n, t_n)$ is in $A$
or $B$. It implies that $(x, t)$ must be a cluster point of points
in $A$ or $B$.
  Let $A_1$, $B_1$ be two subsets of $T(F_i)$, where $$A_1=
  \{(x, t)\in T(F_i)| (x, t) \,\mbox{has unique minimizer},
\,\mbox{and}\, (x, t) \,\mbox {is a cluster point of}\, A\},$$
$$B_1=\{(x, t)\in T(F_i)|(x, t) \,\mbox{has unique minimizer}, \,\mbox{and}\,
(x, t) \,\mbox {is a cluster point of}\, B\}.$$ Hence  $(x, t)$ is in $A_1$ or in $B_1$.
It implies that  $$T(F_i)\cap \{(x,
t)|(x, t)\,\mbox{has unique minimizer}\}=A_1\cup B_1.$$ We also have that $A_1\cap
B_1=\emptyset$, due to the assumption that $distance (A, B)\neq 0$.

Those imply that $(A\cup A_1)\cup (B\cup B_1)=T(F_i)$ and $(A\cup
A_1)\cap (B\cup B_1)=\emptyset$. We also have  $distance(A\cup A_1,
B\cup B_1)\neq 0$,  due to that   $distance
(A, B)\neq 0$ again. Therefore $A\cup A_1$ and $B\cup B_1$ are two
disjoint closed subsets of $T(F_i)$ with $distance(A\cup A_1, B\cup
B_1)\neq 0$. That contradicts with that $T(F_i)$ is connected or
almost connected. Hence the proof is complete.
\end{proof}

Therefore by the above lemma, we have the following theorem
about $\Sigma$:
\begin{theorem}$\label{Sigma}\Sigma$ is composed of at most countable connected
components or almost connected components. That is, $\Sigma=\bigcup
\Sigma_j$.  Each $\Sigma_j$ is a finite or countable union of some
$T(F_i)\setminus \{(x, t)|(x, t)\,\mbox{has unique minimizer}\}$.
\end{theorem}
Although ${\mathbb R^n}\times \{t>0\}\setminus \Sigma$ is not necessarily open, by the similar argument in Proposition 3.3.4 in \cite{C}, we have the following proposition about the regularity of $u$ on ${\mathbb R^n}\times \{t>0\}\setminus \Sigma$.
\begin{proposition}\label{C1} $u$ is $C^1$ on ${\mathbb R^n}\times \{t>0\}\setminus \Sigma$.
\end{proposition}
Next we consider the property about connected components of
$\overline \Sigma$. First we see $\overline S_j$, recall $S_j$ is a
connected component or an almost connect component of $T(M)$. The
following proposition tells us that no matter $S_j$ is  connected or
 almost connected, $\overline S_j$ is a connected set.
\begin{proposition}$\overline S_j$ is connected.
\end{proposition}
\begin{proof}We will prove by contradiction. Assume that
$\overline S_j$ is not connected. Then there exist two disjoint
closed sets $A$ and $B$ of $\overline S_j$ such that $A\cup
B=\overline S_j$. Hence $A\cap  S_j$ and $B\cap S_j$ are two
disjoint closed sets of $S_j$. Since $S_j$ is connected or almost
connected, we have that $distance (A\cap  S_j, B\cap S_j)=0$.
Therefore there exist  some sequence $\{(x_n, t_n)\}\subset A\cap
S_j$, $\{(\bar x_n, \bar t_n)\}\subset B\cap S_j$ and some point
$(x_0, t_0)$ such that $(x_n, t_n)\rightarrow (x_0, t_0)$ and $(\bar
x_n, \bar t_n)\rightarrow (x_0, t_0)$. It implies that $(x_0,
t_0)\in \overline S_j$. Hence  $(x_0, t_0)$ is in $A$ or $B$. If we
assume that it is in $A$, then we have that $(\bar x_n, \bar t_n)\in
B\cap  S_j\subset B$ such that $(\bar x_n, \bar t_n)\rightarrow
(x_0, t_0)\in A$. It contradicts with that $A$ and $B$ are two
closed disjoint sets of $\overline S_j$.
\end{proof}
By Proposition \ref{over}, $\overline\Sigma=\overline
{T(M)}=\overline {\cup S_j}$. In general, $\overline {\cup
S_j}\supset \bigcup \overline S_j$. For any $(x, t)\in \overline
{\cup S_j}\setminus \bigcup \overline  S_j$, it is a cluster point
of points from different $S_j$. This case is quite complicate,
perhaps we could not guarantee that $\overline {\cup S_j}$ has at
most countable connected components. Hence  a simpler
case is considered, which is the case when $\overline {\cup S_j}= \bigcup
\overline  S_j$. For instance, when the number of connected and
almost connected components of $T(M)$ is finite.

Therefore we have the following theorem:
\begin{theorem}\label{overSigma}If $\overline {\cup S_j}= \bigcup \overline
S_j
$, then $\overline\Sigma$ is composed of at most connected
components.
\end{theorem}It is trivial to obtain the following corollary.
\begin{corollary}If $\Sigma$ is composed of finite connected components
or almost connected components, then $\overline\Sigma$ is composed
of finite connected components.
\end{corollary}
Since $\overline \Sigma$ is closed, ${\mathbb R^n}\times \{t>0\}\setminus \overline\Sigma$ is an open set. Similarly, we can have the regularity of $u$ on this open set.
\begin{proposition} $u$ is $C^1$ on ${\mathbb R^n}\times \{t>0\}\setminus \overline\Sigma$.
\end{proposition}
\subsection{Relation between  $D^*u(x, t)$ and minimizers of
$(x, t)$}
Next we discuss the relation between  $D^*u(x, t)$ and minimizers of
$(x, t)$. For the case of strictly convex, superlinear and smooth Hamiltonians, in \cite{C}, Theorem 6.4.9
tells us that there is a one-to-one correspondence between  $D^*u(x,
t)$ and the set of minimizers of $(x, t)$. For the Eikonal equation,
we can see from Theorem \ref{diff} that the one-to-one
correspondence between  $D^*u(x, t)$ and the set of minimizers of
$(x, t)$ does not hold anymore. A differentiable point may
correspond to more than one minimizer, even infinitely many
minimizers.    In \cite{C}, for the Eikonal equation with more
general setting and assumption that $g$ is semiconcave with linear
modulus,  it tells us the following: if $0$ is not in $D^*u(x, t)$,
then there is still a one-to-one correspondence between  $D^*u(x,
t)$ and the set of minimizers of $(x, t)$. But what if  $D^*u(x, t)$
may include $0$? Without assumption that $g$ is semiconcave with
linear modulus, only assume that $g$ is $C^1$, we can have the
following proposition: there is still a one-to-one correspondence,
while not between  $D^*u(x, t)$ and the set of minimizers of $(x,
t)$, but  between  $D^*u(x, t)$ and another set, which is the set of
value $Dg$ at these minimizers of $(x, t)$.
\begin{proposition}\label{*}
There is a one-to-one correspondence between  $D^*u(x, t)$ and
$\tilde L(x, t)=\{Dg(y_0)|y_0\in L(x, t)\}$.
\end{proposition}
\begin{proof}For any $l\in D^*u(x, t)$, by the definition of $D^*u$,
there exists some sequence $\{(x_n, t_n)\}$ such that $(x_n, t_n)
\rightarrow (x, t)$ and $Du(x_n, t_n)\rightarrow l$. Let $y_n\in
L(x_n, t_n)$, since $L$ is upper semi-continuous multi-function,
there exists some subsequence $\{y_{n_k}\}$ and some $y_0$ such that
$y_{n_k}\rightarrow y_0$ and $y_0\in L(x, t)$. By Theorem
\ref{diff}, we have that $\nabla u(x_{n_k}, t_{n_k})=Dg(y_{n_k})$.
Hence $Du(x_{n_k}, t_{n_k})=(Dg(y_{n_k}), -|Dg(y_{n_k})|)$. Let
$k\rightarrow \infty$, we have that $l=(Dg(y_0), -|Dg(y_0)|)$. Hence
$D^*u(x, t)\subset \{(Dg(y_0), -|Dg(y_0)|)|y_0\in L(x, t)\}$.

Next, for any $y_0\in L(x, t)$, there are two cases:

Case 1, $Dg(y_0)\neq 0$. By Proposition \ref{limsup}, we have that $
\limsup_{y_n\rightarrow y_0}\overline t_s(y_n)=t_s(y_0)$. Hence
there exists some sequence $\{y_k\}$ such that $y_k\rightarrow y_0$
and \\$\lim_{y_k\rightarrow y_0}\overline t_s(y_k)=t_s(y_0)$. We have
known that $t\le t_s(y_0)$. Let
$x_k=y_k+\frac{Dg(y_k)}{|Dg(y_k)|}t_k$, where $t_k<\overline
t_s(y_k)$ and $t_k\rightarrow t$. Since $(x_k, t_k)$ has unique
minimizer $y_k$, by Theorem \ref{diff}, $Du(x_k, t_k)=(Dg(y_k),
-|Dg(y_k)|)$. Let $k\rightarrow \infty$,  $Du(x_k, t_k)\rightarrow
(Dg(y_0), -|Dg(y_0)|)$.  Since $(x_k, t_k)\rightarrow (x, t)$, we
have that $l=(Dg(y_0), -|Dg(y_0)|)\in D^*u(x, t)$.

Case 2, $Dg(y_0)=0$. There exists some $P_0$ with $|P_0|\le 1$ such that $x=y_0+P_0t$. Let $t_n<t$ and $t_n\rightarrow t$. Let $x_n=y_0+P_0t_n$.

Claim: For any $y\in L(x_n, t_n)$, we have that $Dg(y)=0$.

Proof of the claim: By contradiction arguments. If assume $\exists\,
\bar y\in L(x_n, t_n)$ with $Dg(\bar y)\neq 0$. Then $g(\bar y)=g(y_0)$
and $\bar y\in \overline B_{t_n}(x_n)\subset B_t(x)$. Hence $\exists y_1\in
B_t(x)$ such that $g(y_1)<g(\bar y)=g(y_0)$. It contradicts with
that $y_0$ is a minimizer for $(x, t)$. The proof of the claim is
complete.

By the claim above and Theorem \ref{diff} 3), $u$ is differentiable
at $(x_n, t_n)$
 and $Du(x_n, t_n)=0\in {\mathbb R}^{n+1}$. It implies that
 $Du(x_n, t_n)\rightarrow 0=(Dg(y_0), -|Dg(y_0|)\in D^*u(x, t)$.

 From above, it implies that \begin{equation}D^*u(x, t)=\{(Dg(y_0),
 -|Dg(y_0)|)\big| \;y_0\in L(x, t)\}.
 \end{equation}
 Next we will prove that there is a one-to-one correspondence between
 \\$\{(Dg(y_0),
 -|Dg(y_0)|)\big| \;y_0\in L(x, t)\}$ and $\tilde L(x, t)=\{Dg(y_0)
 |y_0\in L(x, t)\}$.  It is easy to see that for any
 $y_0$, $y_1$, $(Dg(y_0), -|Dg(y_0)|)=(Dg(y_1), -|Dg(y_1)|)$ if and
  only if $Dg(y_0)=Dg(y_1)$.
\end{proof}
\subsection{ Cause for the formation of the  singularity points}

Finally in this section, we consider the reason for the appearance of the  singularity points. Here a singularity point means a point in $\Sigma\cup T_1$.
 For one dimensional case,  local maximum is the reason for the appearance of the set of singularity points. But for the multi-dimensional case, local maximum is not the only reason. First we show some examples:

{\bf Example 1:} $g$ is a saddle surface with the saddle point at
$0$.

\noindent There is no local maximum for $g$, but for any point
$y_0\in {\mathbb R}^n\setminus\{0\}$, we have  $Dg(y_0)\neq 0$ and
$\exists\, \bar y$, s.t. $\bar y\in H_{y_0}$ and $g(\bar y)<g(y_0)$.
It implies that $\overline t_s(y_0)<\infty$. Therefore the set of singularity points appears. Furthermore $M={\mathbb
R}^n\setminus\{0\}$, which only has one connected component, so by
Theorem \ref{thm1}, $T(M)=\Sigma$ only has one connected component
or one almost connected component (see in Definition \ref{almost}).

{\bf Example 2:} $g$ is a oblique saddle surface.

\noindent There is no local maximum for $g$, furthermore, there is
no point with $Dg=0$. Similar to the  argument in example 1, for any
$y_0\in {\mathbb R}^n$, \, $\exists\, \bar y$, s.t. $\bar y\in
H_{y_0}$ and $g(\bar y)<g(y_0)$. It implies that $\overline
t_s(y_0)<\infty$. Therefore the set of singularity points appears. Every
characteristic touches the set of singularity points. Furthermore, $M={\mathbb
R}^n$, which only has one connected component.  $T(M)=\Sigma$ only
has one connected component or one almost connected component.

Next, we show that when initial data has a strict local maximum, the set of singularity points does appear as one dimensional case.
\begin{proposition}If $y_0$ is a strict local maximum; i.e. there is a bounded and closed set $A\ni y_0$ such that for any $y$ in $A$, $g(y)=g(y_0)$, and there is a bounded and  open set $U\supset A$ such that for   any $y$ in $U\setminus A$, $g(y)<g(y_0)$,  then there exists some $y_1$, such that $\overline t_s(y_1)<\infty$.
\end{proposition}
\begin{proof}There exist some  $y_1$, $y_2$ such that $y_2\in A$ and
$y_1\in U\setminus A$, $y_2\approx y_1$ and $Dg(y_1)\neq 0$. Since
$g(y_2)>g(y_1)$, it is easy to see that $y_2\in H_{y_1}$. By that
$U$ is open and $A\subset U$, there exists some $y_3$ such that
$y_3\in U\setminus A$, $y_3\in H_{y_1}$ and $g(y_3)<g(y_1)$.
\end{proof}

 We can see that local maximum or $Dg=0$ is not only reason to produce
 the set of singularity points.  From the above arguments or by Proposition \ref{M=D}, it can be emphasized that  the reason to the appearance of the set of singularity points is that $\exists\, \bar y$, s.t. $\bar y\in H_{y_0}$ and $g(\bar y)\le g(y_0)$.

\section{$C^2$ initial data case: global structure }
In this section, we consider a better class of initial data, $C^2$
case.  Some better results can be obtained. The tool of curvatures of surfaces can also be applied.
\subsection{Better properties of termination points}
 First we study  properties
 of termination times of characteristics from $y_0$ with $Dg(y_0)\neq 0$.
\begin{proposition}
If $Dg(y_0)\neq 0$, then $\overline t_s(y_0)>0$.
\end{proposition}
\begin{proof}Since $g$ is in $C^2$, then there exists a  $C^2$ $n-1$ dimensional surface
$H(y_0)$ (recall Definition \ref{H}). Let $Dg(y_0)$ be normal vector of $H(y_0)$ at $y_0$.
 The principal  curvatures of $H(y_0)$ at $y_0$ are  $\lambda_{n-1}(y_0)\le ...\le\lambda_2(y_0)\le \lambda_{1}(y_0)$. It is known that all the principal  curvatures of a ball with inner normal vector and radius $t$ are $\frac{1}{t}$.  Then
there exists a $t_0$ such that the maximal principal curvature $\lambda_1(y_0)$ of $H(y_0)$ at $y_0$ with $\lambda_{1}(y_0)<\frac{1}{t_0}$. Hence $\lambda_{n-1}(y_0)\le ...\le\lambda_2(y_0)\le \lambda_{1}(y_0)<\frac{1}{t_0}$.  When $t_0$ is
sufficiently small, we have that 
 $\overline B_{t_0}(x_0)\setminus\{y_0\}\subset\subset H^+(y_0)$, where $x_0=y_0+\frac{Dg(y_0)}{|Dg(y_0)|}t_0$. It implies
that $y_0$ is unique minimizer for $(x_0, t_0)$. Therefore
$\overline t_s(y_0)\ge t_0>0$.
\end{proof}
Contrast to $C^2$ case, when initial data is $C^1$, for some $y_0$ such that $Dg(y_0)\neq 0$, $t_s(y_0)$
may be 0. The reason is the following:  For any $t>0$, let $x=y_0+\frac{Dg(y_0)}{|Dg(y_0)|}t$.  Since $H(y_0)$ is in $C^1$ and $\partial B_t(x)$ is in $C^2$, it is not necessary to exist a  ball $B_t(x)$ such that $H(y_0)$ is in the exterior of $ B_t(x)$. If there does not exist such ball for any $t$, then  $y_0$ is not a minimizer for $(x,t)$. That implies $t_s(y_0)=0$.

Recall in the previous section, we have that for $C^1$ initial data,
\\ $\liminf_{y_n\rightarrow y_0}t_s(y_n)\le \overline t_s(y_0)$. Now
for $C^2$ initial data,  a better result holds.
\begin{proposition}\label{liminf=}$$\liminf_{y_n\rightarrow
y_0}t_s(y_n)=
 \liminf_{y_n\rightarrow y_0}\overline t_s(y_n)= \overline t_s(y_0).
 $$
\end{proposition}
\begin{proof} We only need to prove $\liminf_{y_n\rightarrow y_0}
\overline t_s(y_n)\ge \overline t_s(y_0)$. For any
$t<\overline t_s(y_0)$, let $x=y_0+\frac{Dg(y_0)}{|Dg(y_0)|}t$. For any $y_n\approx y_0$, let $x_n=y_n+\frac{Dg(y_n)}{|Dg(y_n)|}t$, so $x_n\approx x$. Since $y_0$ is unique minimizer for $(x, t)$,  for any $\tilde y_n\in L(x_n, t)$,  $\tilde y_n\approx y_0$ holds. It implies that $\tilde y_n\approx y_n$. Next we want to prove that $y_n$ is unique minimizer for $(x_n, t)$.

Since $g$ is $C^2$,  $H(y_0)$  and $H(y_n)$ are $C^2$ surfaces.  We have that the maximal  principal  curvature $\lambda_{1}(y_0)$ of $H(y_0)$ at $y_0$ is
approximately equal to the  maximal  principal  curvature $\lambda_{1}(y_n)$ of $H(y_n)$ at $y_n$ when $y_n$ is sufficiently close to $y_0$.  Since
$t<\overline t_s(y_0)$, $y_0$ is unique minimizer for $(x, t)$. It implies that $H(y_0)$ is outside of $\overline B_t(x)\setminus\{y_0\}$. Hence $\lambda_{1}(y_0)<\frac{1}{t}$, recall that all the principal  curvatures of a ball with inner normal vector and radius $t$ are $\frac{1}{t}$.  Therefore $\lambda_{1}(y_n)<\frac{1}{t}$; i.e. all the principal curvatures of $H(y_n)$ at $y_n$ are smaller than $\frac{1}{t}$. It implies that in a small neighborhood of $y_n$, $H(y_n)$ is outside of $\overline B_t(x_n)\setminus\{y_n\}$.
 Hence in that small neighborhood of $y_n$,
$\overline B_t(x_n)\setminus\{y_n\}\subset H^+(y_n)$. That is,  for any $y\in \overline B_t(x_n)\setminus\{y_n\}$ and $y\approx y_n$, $g(y)>g(y_n)$ holds.
  Since for $(x_n, t)$, the
minimizer $\tilde y_n\approx y_n$, we have that $y_n=\tilde y_n$, i.e.,
$y_n$ is that unique minimizer. It implies that $\overline
t_s(y_n)\ge t$. Since the above inequality holds for all $y_n\approx
y_0$, we have that
$$\liminf_{y_n\rightarrow y_0}\overline t_s(y_n)\ge t.$$  $t$ is
arbitrary as long as $t<\overline t_s(y_0)$, so
$$\liminf_{y_n\rightarrow y_0}\overline t_s(y_n)\ge \overline
t_s(y_0).$$ Then combine Proposition \ref{liminf} and $\overline
t_s(y_n)\le t_s(y_n)$, the proof is complete.
\end{proof}

 For $C^2$ initial data, we can conclude that
\begin{eqnarray*}\limsup_{y_n\rightarrow y_0} t_s(y_n)&=&
\limsup_{y_n\rightarrow y_0}\overline t_s(y_n)=t_s(y_0)\\
\liminf_{y_n\rightarrow y_0}t_s(y_n)&=& \liminf_{y_n\rightarrow
y_0}\overline t_s(y_n)=\overline t_s(y_0).
\end{eqnarray*}
From their proof, we can see that the above equalities also holds
when $t_s(y_0)=\infty$ or $\overline t_s(y_0)=\infty$.

 From above
equalities, it is easy to see that the following corollary holds:
 \begin{corollary}\label{con}If  $\overline t_s(y_0)=t_s(y_0)$, then
 $\lim_{y_n\rightarrow y_0} t_s(y_n)= \lim_{y_n\rightarrow
 y_0} \overline t_s(y_n)=t_s(y_0)$.
\end{corollary}
Furthermore  the following holds:
 \begin{corollary}If $g$ does not have local minimum, then for any $y_0$ with $Dg(y_0)\neq 0$, we have that   $\overline
t_s(y_0)=t_s(y_0)$.  On
$\{y|Dg(y)\neq 0\}$, $t_s(\cdot)$ is a continuous function. It may be equal to $\infty$ at some
points.
\end{corollary}

Next we consider what will happen when  $\overline
t_s(y_0)<t_s(y_0)$. If $\overline t_s(y_0)<t_s(y_0)$,  by Proposition \ref{cluster}, then the line
segment connecting $(\overline x_s(y_0), \overline t_s(y_0))$ with
$( x_s(y_0),  t_s(y_0))$ belongs to $\Sigma$. Is there a continuous $n-$ dimensional
surface containing this
 line segment that belongs to $\Sigma$ too?
\begin{proposition}\label{surface}For any $y_0$ with $Dg(y_0)\neq 0$,
if $\overline t_s(y_0)<t_s(y_0)$, then on $ H(y_0)$, $\overline
t_s(\cdot)$ and $t_s(\cdot)$ are continuous at $y_0$.
\end{proposition}
Recall that $H(y_0)$ is the $n-1$ dimensional $C^2$ surface $\subset {\mathbb R^n}$ such that $y_0\in H(y_0)$  and for any $y\in H(y_0)$,  $g(y)=g(y_0)$ defined in Definition \ref{H}.
\begin{proof}First by Proposition \ref{liminf=}, we have $$
\liminf_{y_n\in H(y_0), y_n\rightarrow y_0}\overline t_s(y_n)\ge
\liminf_{y_n\rightarrow y_0}\overline t_s(y_n)= \overline
t_s(y_0).$$ Next we need to prove $\limsup_{y_n\in H(y_0),
y_n\rightarrow y_0}\overline t_s(y_n)\le \overline t_s(y_0)$.  For
any $t\in(\overline t_s(y_0), t_s(y_0))$,
$x=y_0+\frac{Dg(y_0)}{|Dg(y_0)|}t$. $L(\overline x_s(y_0), \overline
t_s(y_0))=\{y_0, \tilde y, \hat y,...\}$, where $g$ attains  local
minimum in ${\mathbb R^n}$ at $\tilde y$, $\hat y$,...and these points
are in $\overline B_{\overline t_s(y_0)}(\overline x_s(y_0))$. Due to $Dg(y_0)\neq 0$, $\exists\, \epsilon > 0$ such that $L(\overline x_s(y_0), \overline
t_s(y_0))\setminus\{y_0\}\subset \overline B_{\overline t_s(y_0)}(\overline x_s(y_0))\setminus B_\epsilon (y_0)\subset B_t(x)$. For any
$y_n\in H(y_0)$ and $y_n\approx y_0$, let $x_n=y_n+
\frac{Dg(y_n)}{|Dg(y_n)|}t$. Since $B_t(x_n)\approx B_t(x)$ and $\overline B_{\overline t_s(y_0)}(\overline x_s(y_0))\setminus B_\epsilon (y_0)$ is a closed set,  
 it implies that $B_t(x_n)\supset\overline B_{\overline t_s(y_0)}(\overline x_s(y_0))\setminus B_\epsilon (y_0) $, hence $L(\overline x_s(y_0), \overline
t_s(y_0))\setminus\{y_0\}\subset B_t(x_n)$.

Next we want to prove $y_n\in L(x_n, t)$. Since $\overline t_s(y_0)<t< t_s(y_0)$, similar to the argument in the proof of Proposition \ref{liminf=},  it implies that the maximal principal curvature of $H(y_0)$
at $y_0\le$ the principal curvature of $\partial B_{t_s(y_0)}(x_s(y_0))$ at
$y_0<$ the principal curvature of $\partial B_{t}(x)$ at $y_0$; i.e. $\lambda_1(y_0)\le\frac{1}{t_s(y_0)}<\frac{1}{t}$, so $\lambda_1(y_n)<\frac{1}{t}$; i.e. the maximal principal
curvature of $H(y_n)<$ the principal curvature of $\partial B_{t}(x_n)$ at
$y_n$. It implies that there exists some small $\epsilon>0$ such that  when   $y\in \overline B_{t}(x_n)\cap  B_\epsilon (y_n)$,
$g(y_n)\le g(y)$. When $y\in \overline B_{t}(x_n)\setminus B_\epsilon (y_n)$, since $\overline B_{t}(x_n)\setminus B_\epsilon (y_n)$ is a closed set, $\overline B_{t}(x_n)\setminus B_\epsilon (y_n)\approx \overline B_{t}(x)\setminus B_\epsilon (y_0)$ and  $B_{t_s(y_0)}(x_s(y_0))\supset \overline B_{t}(x)\setminus B_\epsilon (y_0)$,
 we have that $B_{t_s(y_0)}(x_s(y_0))\supset \overline B_{t}(x_n)\setminus B_\epsilon (y_n)$, i.e., $y\in B_{t_s(y_0)}(x_s(y_0))$, hence $ g(y)\ge g(y_0)=g(y_n)$.
Therefore \begin{equation}\label{L(xn)}y_n\in L(x_n,
t).\end{equation}

Since $g(y_n)=g(y_0)=g(\tilde y)=g(\hat y)=...$, we have that $y_n$,
$\tilde y$, $\hat y, ...\in L(x_n, t)$. It implies that $\overline
t_s(y_n)<t$. Since it is true for any $y_n\in H(y_0)$ and
$y_n\approx y_0$,   $\limsup_{y_n\in H(y_0), y_n\rightarrow
y_0}\overline t_s(y_n)\le t$. Therefore $\limsup_{y_n\in H(y_0),
y_n\rightarrow y_0}\overline t_s(y_n)\le \overline t_s(y_0)$. So far
we have that
\begin{equation}\lim_{y_n\in H(y_0), y_n\rightarrow y_0}\overline
t_s(y_n)= \overline t_s(y_0).\end{equation} Next by Proposition
\ref{limsup}, we have that
$$\limsup_{y_n\in H(y_0), y_n\rightarrow y_0} t_s(y_n)\le
\limsup_{y_n\rightarrow y_0} t_s(y_n)=t_s(y_0).$$
Finally we need to prove that $\liminf_{y_n\in H(y_0), y_n\rightarrow
y_0} t_s(y_n)\ge  t_s(y_0)$. For any $t<t_s(y_0)$,
let $x=y_0+\frac{Dg(y_0)}{|Dg(y_0)|}t$. For any $y_n\in H(y_0)$ and
$y_n\approx y_0$, let $x_n=y_n+\frac{Dg(y_n)}{|Dg(y_n)|}t$. 
From (\ref{L(xn)}), we already know $y_n\in L(x_n, t)$. It implies
that $t_s(y_n)\ge t$. Hence $\liminf_{y_n\in H(y_0), y_n\rightarrow
y_0} t_s(y_n)\ge t$. Then $\liminf_{y_n\in H(y_0), y_n\rightarrow
y_0} t_s(y_n)\ge t_s(y_0) $. Therefore
\begin{equation}\lim_{y_n\in H(y_0), y_n\rightarrow y_0} t_s(y_n)=
 t_s(y_0).
\end{equation}
\end{proof}
By Proposition \ref{surface}, we see that for any $y_0$ with $Dg(y_0)\neq 0$ and $\overline t_s(y_0)<t_s(y_0)$, there is a continuous $n-$ dimensional surface through $(\overline x_s(y_0), \overline t_s(y_0))$,  which is a subset of $\Sigma$.

If combine Corollary \ref{con} and Proposition \ref{surface}, on
$H(y_0)$,  the following holds:
\begin{corollary}For any $y_0$ with $Dg(y_0)\neq 0$,   on
$\{y\in H(y_0)|Dg(y)\neq 0, \overline
t_s(y)<\infty\}$,  $\overline t_s(\cdot)$ and $t_s(\cdot)$ are
continuous functions.
\end{corollary}
Proposition \ref{surface} is about what happens  when $y_n\in H(y_0)$.
Next we consider what will happen when $y_n\in H^+(y_0)$ or
$H^-(y_0)$.    Then  the following proposition holds:
\begin{proposition}\label{H+H-}\label{H^+(y_0)}$$\lim_{y_n\in H^+(y_0), y_n
\rightarrow y_0}\overline t_s(y_n)= \lim_{y_n\in H^+(y_0),
y_n\rightarrow y_0} t_s(y_n)=\overline t_s(y_0).$$ And
$$\lim_{y_n\in H^-(y_0), y_n\rightarrow y_0}\overline t_s(y_n)=
 \lim_{y_n\in H^-(y_0), y_n\rightarrow y_0} t_s(y_n)= t_s(y_0).
$$
\end{proposition}
\begin{proof}

By Proposition \ref{limsH+}, we have  that $$\limsup_{y_n\in
H^+(y_0), y_n\rightarrow y_0}\overline t_s(y_n)=\limsup_{y_n\in
H^+(y_0), y_n\rightarrow y_0} t_s(y_n)=\overline t_s(y_0).$$  By
Proposition \ref{liminf=}, $$\liminf_{ y_n\rightarrow y_0}\overline
t_s(y_n)=\liminf_{ y_n\rightarrow y_0}
t_s(y_n)=\overline t_s(y_0)
$$
Hence $\lim_{y_n\in H^+(y_0), y_n\rightarrow y_0}\overline
t_s(y_n)=\lim_{y_n\in H^+(y_0), y_n\rightarrow y_0}
t_s(y_n)=\overline t_s(y_0)$.


Next  we want to prove $\liminf_{y_n\in H^-(y_0),\,y_n\rightarrow
y_0}\overline t_s(y_n)\ge t_s(y_0)$. For any $y_n\in H^-(y_0)$ with
$y_n\approx y_0$,  it is true that  $g(y_n)<g(y_0)$.

Since $g$ is in $C^2$,  $H(y_0)$ and $H(y_n)$ are $C^2$. For any
$t<t_s(y_0)$, we have that $H(y_0)$ is outside of $\overline B_t(x)\setminus \{y_0\}$, hence the maximal principal curvature of $H(y_0)$ at $y_0$ is
smaller than the principal curvature of $\partial B_t(x)$,  so the maximal principal curvature
of $H(y_n)$ at $y_n$ is smaller than the  principal curvature of $\partial
B_t(x_n)$ when $y_n$ is sufficiently close to $y_0$, here
$x=y_0+\frac{Dg(y_0)}{|Dg(y_0)|}t$,
$x_n=y_n+\frac{Dg(y_n)}{|Dg(y_n)|}t$. It implies that when $y\in
\overline B_t(x_n)\setminus \{y_n\}$ and $y\approx y_n$, we have that
$g(y)>g(y_n)$.

For any $y\in \overline B_t(x_n)\setminus \{y_n\}$ and $y$ is not close to $y_n$, $y$ is in $B_{t_s(y_0)}( x_s(y_0))$ when $y_n$ is sufficiently close to $y_0$. It implies that $g(y)\ge g(y_0)>g(y_n)$.
That is, $y_n$ is unique minimizer for $(x_n, t)$. Then we obtain
$$\overline t_s(y_n)\ge t.$$ So $$\liminf_{y_n\in H^-(y_0),\,y_n
\rightarrow y_0}\overline t_s(y_n)\ge t.$$ Since $t$ is arbitrary as
long as $t< t_s(y_0)$, $$\liminf_{y_n\in H^-(y_0),\,y_n\rightarrow
y_0}\overline t_s(y_n)\ge t_s(y_0).$$ By $\overline t_s(y)\le
t_s(y)$, we have that $$\liminf_{y_n\in H^-(y_0),\,y_n\rightarrow
y_0}t_s(y_n)\ge\liminf_{y_n\in H^-(y_0),\,y_n\rightarrow
y_0}\overline t_s(y_n)\ge t_s(y_0).$$ By Proposition \ref{limsup},
$\limsup_{y_n\rightarrow y_0} t_s(y_n)=\limsup_{y_n\rightarrow y_0}
\overline t_s(y_n)=t_s(y_0)$. Hence $\lim_{y_n\in H^-(y_0),
y_n\rightarrow y_0}\overline t_s(y_n)=
 \lim_{y_n\in H^-(y_0), y_n\rightarrow y_0} t_s(y_n)= t_s(y_0)
$.

\end{proof}
 From the proof above, we can see that the proposition above holds
when $\overline t_s(y_0)=\infty$ or $t_s(y_0)=\infty$.

\subsection{Path connected components of $\Sigma\cup T_1$}
Finally, when $g$ is $C^2$,   better results can be obtained than
the case when $g$ is $C^1$ in Theorem \ref{thm1}.  Recall $F_i$ in
Proposition \ref{F_i}.
\begin{proposition}\label{path}If $g$ is $C^2$,  then $T(F_i)$ is path connected.
\end{proposition}
\begin{proof}First we prove that  $T(E_i)$ is path connected. By equation (\ref{E_i}) and Proposition \ref{E=J}, we know that for any $y_0\in E_i$, $t_s(y_0)<\infty$.

For any two points $(x_1, t_1)$, $(x_2, t_2)$ in $T(E_i)$, there
exist $y_1\in L(x_1, t_1)$ and $E_i$, $y_2\in L(x_2, t_2)$ and $E_i$
such that $t_1\in [\overline t_s(y_1),  t_s(y_1)]$ and $t_2\in
[\overline t_s(y_2), t_s(y_2)]$. 1. If $y_1=y_2$, then $(x_1, t_1)$
and $(x_2, t_2)$ are connected by the line segment which is a subset
of $\{(x, t)|x=y_1+\frac{Dg(y_1)}{|Dg(y_1)|}t, t\in [\overline
t_s(y_1), t_s(y_1)]\}$.

2. If $y_1\neq y_2$, since $E_i$ is path connected, there is a
continuous curve $\bar h(s)\subset E_i$ with $\bar h(0)=y_1$ and
$\bar h(1)=y_2$. On the basis of $\bar h(s)$, we can  construct a
continuous curve $h(s)$ still with $ h(0)=y_1$ and $ h(1)=y_2$ such
that for each $s_0\in (0, 1]$, there exists some $s_1<s_0$ with $\{
h(s)|s\in [s_1, s_0)\}\subset H( h(s_0))$ or  $\{ h(s)|s\in [s_1,
s_0)\}\subset H^+( h(s_0))$ or  $\{ h(s)|s\in [s_1, s_0)\}\subset
H^-( h(s_0))$ and for each $s_0\in [0, 1)$, there exists some
$s_2>s_0$, similar properties hold. If $\bar h(s)$ satisfies
properties above, then this $\bar h(s)=h(s)$. Otherwise, if for some
$s_0$ such that there does not exist any $s_1<s_0$ with $\{ \bar
h(s)|s\in [s_1, s_0)\}$ included in $H( \bar h(s_0))$ or $H^+( \bar
h(s_0))$ or $H^-( \bar h(s_0))$, then we can construct $h(s)$ as the
following: for any $ s_m\approx s_0$ and $ s_m< s_0$, we have that
$\bar h(s_m)\in H( \bar h(s_0))$ or $H^+( \bar h(s_0))$ or $H^-(
\bar h(s_0))$, since in any neighborhood of $\bar h(s_0)$, $H( \bar
h(s_0))$ separates it into these three parts. 1. If $\bar h(s_m)\in
H( \bar h(s_0))$, then let $h(s)\in H( \bar h(s_0))$ for any $s\in
(s_m, s_0)$. 2. $\bar h(s_m)\in H^+( \bar h(s_0))$. Since $H(y)$ is
some kind of level set along the direction $Dg(y)$, there will be a
continuous curve $h(s)$ with $s\in (s_m, s_0)$ transverse those
level sets such that $g(h(s))$ strictly decreases. 3. If $\bar
h(s_m)\in H^-( \bar h(s_0))$, then similarly to 2, there will be a
continuous curve $h(s)$ with $s\in (s_m, s_0)$ transverse those
level sets such that $g(h(s))$ strictly increases. The case of $
s_m> s_0$ is same as the case of $ s_m<s_0$. By this way, we can
have the desirable $h(s)$, $s\in [0, 1]$.

For any $s_0\in [0, 1]$, 1) if $\overline t_s(h(s_0))=t_s(h(s_0))$,
by Corollary \ref{con}, we have that $\lim_{s\rightarrow
s_0}\overline t_s(h(s))=\lim_{s\rightarrow s_0}
t_s(h(s))=t_s(h(s_0))$. It implies that $\overline t_s(h(s))$ and $
t_s(h(s))$ are both continuous at $s_0$. 2) If $\overline
t_s(h(s_0))<t_s(h(s_0))$, then

Case 1. $\exists s_1,\, s_2$, such that $s_0\in [s_1, s_2]$, ($s_1$
may be equal to $s_2$),  and $\{h(s)|s\in[s_1, s_2]\}\subset$
$H(h(s_1))$. And $\exists s_3,\, s_4$ such that $\{h(s)|s\in(s_3,
s_1]\}\subset H^+(h(s_1))$ and  $\{h(s)|s\in[s_2, s_4)\}\subset H^+
(h(s_2))$. Then by Proposition \ref{H^+(y_0)}, we know that
\begin{eqnarray*}&&\lim_{s\rightarrow s_1^-}\overline t_s(h(s))=
\lim_{s\rightarrow s_1^-} t_s(h(s))= \overline t_s(h(s_1)),\\
&&\lim_{s\rightarrow s_2^+}\overline t_s(h(s))=
\lim_{s\rightarrow s_2^+} t_s(h(s))= \overline t_s(h(s_2)).
\end{eqnarray*}
If $s_1=s_2$; i.e. $s_0=s_1=s_2$, then  $\overline t_s(h(s))$ and $
t_s(h(s))$ are both continuous at $s_0$. If $s_1<s_2$, then by
Proposition \ref{surface}, $$T_1=\{(x, t)|x=
h(s)+\frac{Dg(h(s))}{|Dg(h(s))|}t,  t\in[\overline t_s(h(s)),
t_s(h(s)], s\in [s_1, s_2] \}$$ is a continuous $n$ dimensional
surface which is a subset of $T(E_i)$, it is obvious to be able to
find a curve  on $T_1$ connecting $(\overline x_s(h(s_1)), \overline
t_s(h(s_1)))$ and $(\overline x_s(h(s_2)), \overline t_s(h(s_2)))$.

Case 2. $\exists s_1,\, s_2$, such that $s_0\in [s_1, s_2]$, ($s_1$
may be equal to $s_2$),  and $\{h(s)|s\in[s_1, s_2]\}\subset$
$H(h(s_1) )$. And $\exists s_3,\, s_4$ such that $\{h(s)|s\in(s_3,
s_1]\} \subset H^+(h(s_1))$ and  $\{h(s)|s\in[s_2, s_4)\}\subset
H^-(h(s_2))$. Then by Proposition \ref{H^+(y_0)}, we know that
\begin{eqnarray*}&&\lim_{s\rightarrow s_1^-}\overline t_s(h(s))=
\lim_{s\rightarrow s_1^-} t_s(h(s))= \overline t_s(h(s_1)),\\
&&\lim_{s\rightarrow s_2^+}\overline t_s(h(s))=
\lim_{s\rightarrow s_2^+} t_s(h(s))=  t_s(h(s_2)).
\end{eqnarray*}
If $s_1=s_2$; i.e. $s_0=s_1=s_2$, then  $\overline t_s(h(s))$ and $ t_s(h(s))$ are both continuous from left and from right at $s_0$. $(\overline x_s(h(s_1)), \overline t_s(h(s_1)))$ and $( x_s(h(s_1)),  t_s(h(s_1)))$ can be connected by characteristic segment from $h(s_1)$ . If $s_1<s_2$,
 then it is obvious to be able to find a curve  on $T_1$ connecting $(\overline x_s(h(s_1)), \overline t_s(h(s_1)))$ and $( x_s(h(s_2)),  t_s(h(s_2)))$.

 Similarly, for the left two cases,  the similar  results can be obtained. Therefore $T(E_i)$ is path connected.

 Next for any point $(x, t)$ in $T(F_i\setminus E_i)$,
 $\exists y_0\in F_i\setminus E_i$, such that $x=y_0+
 \frac{Dg(y_0)}{|Dg(y_0)|}t$ and $\overline t_s(y_0)\le t <t_s(y_0)
 =\infty$. By the argument in the proof of Proposition \ref{F_i}, we
 have known that there exist some $U$ which is  a neighborhood of $y_0$
    such that $U\cap H^+(y_0)\subset E_i$. For any point $y_1$ in
    $U\cap H^+(y_0)$, there is a continuous
    curve $h(s)$
     connecting $y_0$ and $y_1$ such that $h(0)=y_0$, $h(1)=y_1$ and
     $\{h(s)|s\in (0, 1]\}\subset U\cap H^+(y_0)$. By Proposition
     \ref{H^+(y_0)},
     we have that $\lim_{s\rightarrow 0^+}\overline t_s(h(s))=
     \lim_{s\rightarrow 0^+} t_s(h(s))= \overline t_s(y_0)$. Then
     combine the argument in the proof of that $T(E_i)$ is path
     connected, it is true that
$(x_s(h(s)), t_s(h(s)))$ and $(\overline x_s(h(s)), \overline
t_s(h(s)))$ are both continuous curves on $[0, 1]$. Finally it is
easy to see that $(\overline x_s(y_0), \overline t_s(y_0))$ is
connected with $(x, t)$ by characteristic segment which is a subset
of $T(F_i\setminus E_i)$. Therefore $T(F_i)$ is path connected.
\end{proof}
So combine Proposition \ref{F_i} and $T(M)=\Sigma\cup T_1$, the following holds:
\begin{theorem}\label{path connected} If $g$ is $C^2$,  then the set of singularity points $\Sigma\cup T_1$
is composed of at
most countable path connected components, that is, $\Sigma\cup T_1=\bigcup
T(F_i)=\bigcup S_j$, where $S_j$ is a path connected component of
$\Sigma\cup T_1$.
\end{theorem}
$T(F_i)$ is path connected, but not necessarily a path connected
component of $T(M)$. $T(F_i)$ may intersect with other $T(F_j)$.  A
path connected component $S_k$ may include one or many $T(F_i)$.

Similar to the $C^1$ initial data case, on the other hand, $F_i$ is path connected, but not necessarily a
path connected component of $M$. $R_j$ is a  path connected
component of $M$. Then there exist some $F_l$, such that
$R_j=\bigcup F_l$, so $T(R_j)=\bigcup T(F_l)$. Hence $T(R_j)$ is not
necessarily path connected, it may include several  path connected
components. The reason is that for any $y_0\in F_l\setminus E_l$, we
have that $\overline t_s(y_0)<t_s(y_0)=\infty$. By Proposition
\ref{H+H-}, $\lim_{y_n\in H^+(y_0), y_n \rightarrow y_0}\overline
t_s(y_n)= \lim_{y_n\in H^+(y_0), y_n\rightarrow y_0}
t_s(y_n)=\overline t_s(y_0)$ and $\lim_{y_n\in H^-(y_0),
y_n\rightarrow y_0}\overline t_s(y_n)= \infty$, $T(F_l\setminus
E_l)$ may separate different path connected components. For one
dimensional case, it is easy to see that $T(F_l\setminus E_l)$ must
separate two neighboring  path connected components.

So far by Theorem \ref{path connected}, we can easily obtain a
result about $\Sigma $ for one dimensional case.
\begin{corollary}\label{one}For one dimensional case, the set of all
nondifferentiable points  $\Sigma$ is composed of at most countable
path connected components.
\end{corollary}
Before giving the proof, we give a lemma first.
\begin{lemma}\label{n=1}For one dimensional case, $$T(M)\cap \{(x, t)|(x, t)\,
\mbox{has unique minimizer}\}=\emptyset.$$
\end{lemma}
\begin{proof}For any $y_0\in M$, 1. if $\overline t_s(y_0)<
t_s(y_0)$, then by Proposition \ref{cluster}, for any $(x, t)$ on
the characteristic from $y_0$ and $\overline t_s(y_0)\le t\le
t_s(y_0)$, $(x, t)$ has more than one minimizer. 2. If $\overline
t_s(y_0)= t_s(y_0)$, for any $(x, t)$ on the characteristic from
$y_0$ and $t>t_s(y_0)$, $(x, t)$ does not have $y_0$ as its
minimizer. Let $y_t\in L(x, t)$. 1) If $Dg(y_t)=0$, since
$Dg(y_0)\neq 0$, there exists an open set $U$ such that $y_0\in U$
and for any $y\in U$, $Dg(y)\neq 0$, then   $y_t$ is  not in $U$. 2)
If $Dg(y_t)\neq 0$, due to that it is one dimensional case, the
characteristic from $y$ with $Dg(y)\neq 0$ only has two choices of
slope, $1$ or $-1$. We must have
$\frac{Dg(y_t)}{|Dg(y_t)|}=-\frac{Dg(y_0)}{|Dg(y_0)|}$. Otherwise,
$y_t=y_0$, it contradicts with that $t>t_s(y_0)$.
 Hence  $|y_t-y_0|=2t$. When $t$ goes to $t_s(y_0)$, for both cases 1) and
2), there exists some $\bar y$ such that $\bar y$ is a minimizer for
$(x_s(y_0), t_s(y_0))$ and $\bar y\neq y_0$. Therefore $(x_s(y_0),
t_s(y_0))$ has more than one minimizer.
\end{proof}
Now we complete the proof of Corollary \ref{one}.
\begin{proof}By the above lemma, $\Sigma=T(M)\setminus \{(x, t)|(x, t)\,
\mbox{has unique minimizer}\}=T(M)$, hence by Theorem \ref{path
connected}, $\Sigma$ is compose of at most countable path connected
sets; i.e. $\Sigma=\bigcup T(F_i)$. Therefore it is composed of at
most countable path connected components.
\end{proof}
\section{$C^k$ initial data case, $k\ge 2$: regularity properties}
In section 3, when the initial data is in $C^1$, it is proved that $u$ is in $C^1({\mathbb R^n}\times \{t>0\}\setminus \Sigma)$.
In this section, we want to find out that when the initial data is in $C^k$, $k\ge 2$,  what are the regularity properties of solutions $u$. First $T_1$ is further studied, the relation $T_1=\Gamma\setminus \Sigma$ is identified, where the set of conjugate points $\Gamma$ is only defined for $C^k$ initial data, $k\ge 2$. Then it is shown that there is  a blow-up of second derivatives of $u$ at any point in $T_1$. For any point before the first termination point on the characteristic from $y_0$ with $Dg(y_0)\neq 0$,  we  find a neighborhood $U$ of it,  $u$ is $C^k(U)$, $k\ge 2$. For characteristics from points with $Dg=0$,  a subtle region $P_0$ is detected. Therefore if $\Sigma\cup T_1\cup P_0$ is cut off, then for any point, there is a  neighborhood $U$ of it,  $u$ is $C^k(U)$,  $k\ge 2$.

  For the case of strictly convex, superlinear and smooth Hamiltonians, it is proved that  if $\Sigma\cup T_1$ is cut off, then $u$ has $C^k$ regularity, $k\ge 2$ (see \cite{Flem}\cite{ZTW})\cite{C}).  The argument in the proof in \cite{ZTW} uses that the range of minimizers  is in ${\mathbb R^n}$, so the criterion of a local minimum in ${\mathbb R^n}$ can be used. While for the Eikonal equation,  the range of minimizers for $(x, t)$ is $\overline B_t(x)$, if $y_0$ is a minimizer for $(x, t)$ with $Dg(y_0)\neq 0$, then $y_0$ is on the boundary of $\overline B_t(x)$, so it is not an interior point of $\overline B_t(x)$. The argument in the proof  for the case of strictly convex, superlinear and smooth Hamiltonians can not be applied here. But the idea of the usage of some kind of linearity for the characteristic can be borrowed from \cite{ZTW}. First we  prove that except  $n$ moments on the effective characteristic, $u$ is $C^k$ in the  neighborhood of any  point.
   Then two important eigenvalues of $X_y$ are identified, and  the desirable result is obtained.

Recall that the characteristic from $y_0$ with $Dg(y_0)\neq 0$ has the following form:
\begin{equation}\label{X}X(y_0,
t)=y_0+\frac{Dg(y_0)}{|Dg(y_0)|}t.\end{equation}
Since in this section, $C^k$ initial data are considered, $k\ge 2$,    the
derivative of $X$ (in definition \ref{X}) in term of $y$ coordinate
is meaningful. Similarly, we follow the definition in
\cite{C}:
\begin{definition}(Conjugates points)

A point $(x, t)$ is called conjugate if there exists $y\in {\mathbb R^n}$
 such that $Dg(y)\neq 0$, $x=X(y,
t)=y+\frac{Dg(y)}{|Dg(y)|}t$,  $y$ is a minimizer for $(x, t)$,  and
$$det X_y(y, t)=0.$$
We denote  by
$\Gamma$ the set of conjugate points.
\end{definition}
Next we want to study the relation between $T_1$ and $\Gamma$.
\begin{proposition}\label{1}If $Dg(y_0)\neq 0$ and $(x_s(y_0), t_s(y_0))$ has
unique minimizer $y_0$, then $det X_y(y_0, t_s(y_0))=0$.
\end{proposition}
\begin{proof}It will be proved by a contradiction argument. Assume
that \\$det X_y(y_0, t_s(y_0))\neq 0$. Define a map $f: (y,
t)\rightarrow (X(y, t), t)$. Since $g$ is $C^2$, we have that $f$ is
a $C^1$ map. By assumption that $det X_y(y_0, t_s(y_0))\neq 0$ and
the inverse function theorem, there exist two  open subsets $V$ and
 $W$ of ${\mathbb R^{n+1}}$ such that $(y_0, t_s(y_0))\in V$ and
 $(x_s(y_0), t_s(y_0))\in W$, and the map $f$ from $V$ to $W$ is
 bijective. Since $(x_s(y_0), t_s(y_0))$ has unique minimizer $y_0$,
for any $(x, t)\approx (x_s(y_0), t_s(y_0))$ and $y\in L(x, t)$,
 we have that $y\approx y_0$. It implies that when $(x, t)$ is
 sufficiently close to $(x_s(y_0), t_s(y_0))$, for any $y\in L(x,
 t)$, $y\in V$. Due to the bijection of $f$ from $V$ to $W$, $y$ is
 unique minimizer for $(x, t)$. By Theorem \ref{diff},  $(x, t)$ is a
 differentiable point of $u$. It contradicts with that
 $(x_s(y_0), t_s(y_0))$ is a cluster point of nondifferentiable
 points of $u$ shown in Proposition \ref{cluster}.
\end{proof}
Hence the proposition above immediately implies
$T_1\subset \Gamma$. We will prove $T_1=\Gamma\setminus \Sigma$ later.

Combining Proposition \ref{cluster} and  the proposition above, we have
that a termination point of the characteristic from $y_0$ with
$Dg(y_0)\neq 0$ is either a nondifferentiable point or a conjugate
point. How about the points on the characteristic  before the first
termination point $(\overline x_s(y_0), \overline t_s(y_0))$? We
already know that these points are differentiable points of $u$. So
the question is that  whether  there are some conjugate points i.e., $det X_y(y_0, t)=0$ holds for some $t<\overline
t_s(y_0)$.

First the following lemma helps us to rule out most points.
\begin{lemma}\label{X_y at most}When $Dg(y_0)\neq 0$, on the characteristic from $y_0$, there are at most $n$ moments $t_i$, $i=1, 2...,n$ such that $det X_y(y_0, t_i)=0$.
\end{lemma}
\begin{proof}Since $X(y_0, t)=y_0+\frac{Dg(y_0)}{|Dg(y_0)|}t$, \begin{equation}\label{X_y}X_y(y_0, t)=I_d+A(y_0)t,\end{equation} where $A(y_0)=(\frac{z}{|z|})'|_{z=Dg(y_0)}D^2g(y_0)$.

If $det X_y(y_0, t_0)=0$, then there exists some $\theta\in {\mathbb R^n}$ such that $$X_y(y_0, t_0)\theta=\theta + A(y_0)\theta t_0=0.$$
It implies that $$A(y_0)\theta=-\frac{\theta}{t_0}.$$
Hence $-\frac{1}{t_0}$ is an eigenvalue of $A(y_0)$ with eigenvector $\theta$. Since $A(y_0)$ has at most $n$ real eigenvalues, we complete the proof.
\end{proof}
Next we find out that by using the lemma above and two special eigenvectors, it is proved that on the characteristic from $y_0$ with $Dg\neq 0$, before the first termination time, actually, every point is not a conjugate point.
This result is similar to the case of strictly convex, superlinear and smooth Hamiltonians, and it is worth mentioning that for the case of strictly convex, superlinear and smooth Hamiltonians, there is only one termination time for every
characteristic, in stead of possible two termination time for the
Eikonal equation.

\begin{theorem}\label{det}For any $y_0$ with $Dg(y_0)\neq 0$, and any $t<\overline t_s(y_0)$, we have that  $det X_y(y_0, t)\neq 0$.
\end{theorem}
\begin{proof}We will prove by a contradiction argument. Assume that there exists some $t<\overline t_s(y_0)$ such that $det X_y(y_0, t)= 0$. Let $x=y_0+\frac{Dg(y_0)}{|Dg(y_0)|}t$. Since $t<\overline t_s(y_0)$, by Lemma \ref{diff}, $u$ is differentiable at $(x, t)$ and $$u(x, t)=g(y_0).$$
Take differentiation in term of $y$, then $$\nabla u(x, t)X_y(y_0, t)=Dg(y_0).$$
By Lemma \ref{diff}, we know that $\nabla u(x, t)=Dg(y_0)$, hence we have that
$$Dg(y_0)X_y(y_0, t)=Dg(y_0).$$
It implies that $Dg(y_0)$ is left eigenvector of $X_y(y_0, t)$ with eigenvalue 1. Accordingly, there exists a vector $r_1$, which is a right eigenvector of $X_y(y_0, t)$ with eigenvalue 1.

Next due to $det X_y(y_0, t)= 0$, there exists $\theta$ such that $$X_y(y_0, t)\theta=0.$$ Hence $\theta$ is a right eigenvector of $A(y_0)$ in (\ref{X_y}) with eigenvalue $-\frac{1}{t}$.

By Lemma \ref{X_y at most}, we know that there are at most $n$ moments $t_i$
such that $det X_y(y_0, t_i)= 0$, so there is some $\tilde t\in (t, \overline t_s(y_0))$ such that $det X_y(y_0, \tilde t)\neq 0$.

Due to that the length of $\theta$ and $r_1$ is arbitrary, we can let $|\theta|=|r_1|<<1$. First observe that
\begin{eqnarray*} X_y(y_0, \tilde t)ar_1&=&ar_1, \qquad a\in (-1, 1)\\
 X_y(y_0, \tilde t)\theta &=&-\frac{\tilde t-t}{t}\theta.
\end{eqnarray*}
Let $l$ denote the line segment from $y_0+\theta$ to $X(y_0, \tilde t)-\frac{\tilde t-t}{t}\theta$. Let $AB$ denote the parallelogram with four vertices $y_0-r_1$, $y_0+r_1$, $X(y_0, \tilde t)-r_1$ and $X(y_0, \tilde t)+r_1$. It is easy to see that the characteristic segment from $y_0$ and $t\in [0, \tilde t]$ is the mid line of $AB$.  Since $-\frac{\tilde t-t}{t}$ is negative, $\theta$ and $r_1$ are independent, we have that $l$ must intersect with $AB$ into the point $(X(y_0, t), t)$ on the mid line. Next
\begin{eqnarray*}X(y_0+ar_1, \tilde t)&=&X(y_0, \tilde t)+X_y(y_0, \tilde t)ar_1+o(|r_1|)\\&=&X(y_0, \tilde t)+ar_1+o(|r_1|),\\
X(y_0+\theta, \tilde t)&=&X(y_0, \tilde t)+X_y(y_0, \tilde t)\theta+o(|\theta|)\\&=&X(y_0, \tilde t)-\frac{\tilde t-t}{t}\theta+o(|\theta|).
\end{eqnarray*}
Let $\tilde l$ denote the characteristic segment from $y_0+\theta$, $t\in [0, \tilde t]$. Let $\tilde {AB}$ denote the set containing all characteristic segments from $y_0+ar_1$, $a\in (-1, 1)$, $t\in [0, \tilde t]$. By the equalities above,   the distance between $l$ and $\tilde l$ and the distance between $AB$ and $\tilde {AB}$ are $o(|r_1|)$; i.e. $o(|\theta|)$. It can be considered as a very small disturbance of order of $o(|\theta|)$. Hence $\tilde {AB}$ and $\tilde l$ must intersect with each other into one point as well. Let $$(\hat x, \hat t)=\tilde {AB}\cap \tilde l, $$we have that $\hat t<\tilde t<\overline t_s(y_0)$. $(\hat x, \hat t)$ is on the characteristic from $y_0+\theta$ and on the characteristic from $y_0+a_0r_1$ for some $a_0\in (-1, 1)$. That is,
\begin{eqnarray*}\hat x&=&y_0+\theta+\frac{y_0+\theta}{|y_0+\theta|}\hat t\\
&=&y_0+a_0r_1+\frac{y_0+a_0r_1}{|y_0+a_0r_1|}\hat t.
\end{eqnarray*}
It implies that $y_0+\theta$ and $y_0+a_0r_1$ are on the $\partial B_{\hat t}(\hat x)$. Then there are four cases: 1. If $g(y_0+\theta)=g(y_0+a_0r_1)$ and   $y_0+\theta$, $y_0+a_0r_1$ are minimizers for $(\hat x, \hat t)$, then  $(\hat x, \hat t)$ has more than one minimizer with $Dg\neq 0$, hence   $\hat t= \overline t_s(y_0+\theta)=\overline t_s(y_0+a_0r_1)$. 2.  If $g(y_0+\theta)=g(y_0+a_0r_1)$ and $y_0+\theta$, $y_0+a_0r_1$ are not minimizers for $(\hat x, \hat t)$,  then $\hat t> \overline t_s(y_0+\theta)$ and $\hat t>\overline t_s(y_0+a_0r_1)$. 3. If $g(y_0+\theta)<g(y_0+a_0r_1)$, then  $\hat t>\overline t_s(y_0+a_0r_1)$. 4. If $g(y_0+\theta)>g(y_0+a_0r_1)$, then $\hat t> \overline t_s(y_0+\theta)$.

 From above, we can conclude that $\overline t_s(y_0+\theta)\le \hat t$ or $\overline t_s(y_0+a_0r_1)\le \hat t$. Hence $\overline t_s(y_0+\theta)\le \tilde t$ or $\overline t_s(y_0+a_0r_1)\le \tilde t$. Due to that $|\theta| $ and $|r_1|$ can be arbitrarily small, the follwing holds, $$\underline{\lim}_{y_n\rightarrow y_0} \overline t_s(y_n)\le \tilde t<\overline t_s(y_0).$$ It contradicts with that $\underline{\lim}_{y_n\rightarrow y_0} \overline t_s(y_n)=\overline t_s(y_0)$ in Lemma \ref{liminf}. Hence we obtain the desired result.
\end{proof}
By the theorem above, we immediately know that any point in $\Gamma$ must be a termination point.   With Proposition \ref{1},    the following holds:
\begin{corollary}$T_1=\Gamma\setminus\Sigma$.
\end{corollary}
That is, the following equivalence  holds:
\begin{corollary}\label{3}If $(x_0, t_0)$ has unique minimizer $y_0$ with $Dg(y_0)\neq 0$, then $det X_y(y_0, t_0)=0$ iff $t_0=\overline t_s(y_0)=t_s(y_0)$.
\end{corollary}
Next similar to the proof of Theorem 6.4.10 in \cite{C}, the following holds:
\begin{corollary}\label{2}For any $y_0$ with $Dg(y_0)\neq 0$ and any $t_0 < \overline t_s(y_0)$, there exists a neighborhood $W$ of $(X(y_0, t_0), t_0)$ such that $u$ is $C^k$ on $W$.
\end{corollary}
\begin{proof}Recall the map $f: (y, t)\rightarrow (X(y, t), t)$. By Theorem \ref{det}, for any $y_0$ with $Dg(y_0)\neq 0$ and any $t_0 < \overline t_s(y_0)$, $det X_y(y_0, t_0)\neq 0$ holds. By the inverse function theorem, there exist two open sets $V$, $W$ of ${\mathbb R^{n+1}}$ such that $(y_0, t_0)\in V$ and
 $(X(y_0, t_0), t_0)\in W$, and the map $f$ from $V$ to $W$ is
 one- to- one and onto and a $C^{k-1}$ diffeomorphism. For any $(x,t)\in W$, let $(y, t)=f^{-1}(x, t)$. Hence $(y, t)\in V$.   When $W$ is sufficiently small, due to  Proposition \ref{liminf=} that $\underline {\lim}_{y_n\rightarrow y_0}\overline t_s(y_n)=\overline t_s(y_0)$,  we have that $t<\overline t_s(y)$. It implies that $y$ is unique minimizer for $(x, t)$. Hence by Theorem \ref{diff}, $u$ is differentiable at $(x, t)$ and $\nabla u(x,t)=Dg(y)=Dg(\Pi_y f^{-1}(x, t))$. Therefore $\nabla u$ is $C^{k-1}$ on $W$. $u_t(x, t)=-|\nabla u(x, t)|$, since $\nabla u(x, t)=Dg(y)\neq 0$, we have that $u_t$ is $C^{k-1}$ on $W$. Therefore $u$ is $C^k$ on $W$.
\end{proof}

For any $(x_0, t_0)\in T_1$, we have known that it is a differentiable point of $u$. Can it be considered as a "good" point? From Proposition \ref{cluster}, it is a cluster point of nondifferentiable points from later time. Next proposition tells us that the norm of second derivatives is going to blow up when $t\rightarrow t_0^-$ along the characteristic. So roughly speaking, the points in $T_1$ are generating points of  the nondifferentiable points of $u$. Theorem 6.5.4 in \cite{C} shows the same property for the case of strictly convex, superlinear and smooth Hamiltonians.
\begin{proposition} Let $(x_0, t_0)\in T_1$ and $y_0$ be the unique minimizer for $(x_0, t_0)$, then $$\lim_{t\rightarrow t_0^-}||\nabla^2 u(X(y_0, t), t)||=+\infty.$$
\end{proposition}
\begin{proof}Since $(x_0, t_0)\in T_1$, $t_0=t_s(y_0)=\overline t_s(y_0)$. When $t<\overline t_s(y_0)$, by Corollary \ref{2}, we know that there is a neighborhood $W$ of $(X(y_0, t), t)$, $u\in C^k(W)$. By Lemma \ref{diff},
$$\nabla u(X(y_0, t), t)=Dg(y_0),$$hence
\begin{equation}\label{nabla2}\nabla^2 u(X(y_0, t), t)X_y(y_0, t)=Dg^2(y_0).\end{equation}
Recall $X(y_0, t)=y_0+\frac{Dg(y_0)}{|Dg(y_0)|}t$ and
$$X_y(y_0, t)=I_d+K'(Dg(y_0))D^2g(y_0)t, $$ where $K(z)=\frac{z}{|z|}$.
By Corollary \ref{3}, $det X_y(y_0, t_0)=0$. It implies that there exists some $\theta\in {\mathbb R^n}\setminus \{0\}$ such that $$0=X_y(y_0, t_0)\theta=\theta+K'(Dg(y_0))D^2g(y_0)\theta t.$$
From it, we know that $D^2g(y_0)\theta\neq 0$. Go back to (\ref{nabla2}), then $$\nabla^2 u(X(y_0, t), t)X_y(y_0, t)\theta=Dg^2(y_0)\theta.$$
It implies that $$||\nabla^2 u(X(y_0, t), t)||\ge \frac{|Dg^2(y_0)\theta|}{|X_y(y_0, t)\theta|}\rightarrow +\infty$$
as $t\rightarrow t_0^-$ due to $0=X_y(y_0, t_0)\theta$.
\end{proof}
Actually, by the argument above, for any $(x_0, t_0)\in \Gamma$, let $y_0$ be a minimizer for $(x_0, t_0)$ with $det X_y(y_0, t_0)=0$, then $\lim_{t\rightarrow t_0^-}||\nabla^2 u(X(y_0, t), t)||=+\infty$ also holds.

Next we take a look at characteristics from some point $y_0$ with $Dg(y_0)=0$. First when $|P|<1$, we have the following:
\begin{proposition}\label{P<1}For any point $y_0$ with $Dg(y_0)=0$, 1) if there exists some $P$ such that $|P|<1$ and $t_s(y_0,P)>0$, then $g$ attains a local minimum at $y_0$. 2) For any point $(x_0, t_0)$, where $t_0<t_s(y_0,P)$, $x_0=X(y_0, P, t_0)=y_0+Pt_0$ and $|P|<1$, there exists a neighborhood $U$ of $(x_0, t_0)$ such that $u\equiv g(y_0)$ on $U$. 3) For $(x_s(y_0,P), t_s(y_0,P))$ with $|P|<1$, it is either a nondifferentiable point of $u$ or a point on the characteristic from another point $y_1$ with $Dg(y_1)=0$ and with some direction $P_1$ such that $|P_1|=1$.
\end{proposition}
\begin{proof}Since $y_0$ is a minimizer for $(x_s(y_0,P), t_s(y_0,P))$,  for any $y $ in the ball $\overline B_{t_s(y_0,P)}(x_s(y_0,P))$, $g(y_0)\le g(y)$. Due to $|P|<1$, we have that $y_0$ is inside the ball,  hence $g$ attains a local minimum at $y_0$.

Next we will prove 2). For any $y$ in the ball $\overline B_{t_s(y_0,P)}(x_s(y_0,P))$, $g(y_0)\le g(y)$. Since $|P|<1$, $y_0$ is an interior point of $B_{t_0}(x_0)$. There exists a neighborhood $U$ of $(x_0, t_0)$ such that $U$ is inside the cone with the vertex $(x_s(y_0,P), t_s(y_0,P))$ and the base $\overline B_{t_s(y_0,P)}(x_s(y_0,P))$, and  for any $(x, t)\in U$, $y_0\in B_t(x) $. Hence  $\overline B_t(x)\subset \overline B_{t_s(y_0,P)}(x_s(y_0,P))$. Therefore $y_0$ is a minimizer for $(x, t)$. We conclude that  $u(x, t)\equiv g(y_0)$ when $(x, t)\in U$.

Finally we will prove 3). For the second termination point $(x_s(y_0,P), t_s(y_0,P))$, it must have another minimizer, otherwise it will contradict with 1).  If there is a minimizer $y_1$ with $Dg(y_1)\neq 0$, then $(x_s(y_0,P), t_s(y_0,P))\in \Sigma$. If there is no minimizer with $Dg\neq 0$, then there is a minimizer $y_1$ with $Dg(y_1)= 0$, and $(x_s(y_0,P), t_s(y_0,P))$ lies in the characteristic from $y_1$ with some direction $P_1$ such that  $|P_1|=1$. Otherwise, it will contradicts with 1) and that $(x_s(y_0,P), t_s(y_0,P))$ is the second termination point of the characteristic from $y_0$ with the direction $P$.
\end{proof}
For the case of strictly convex, superlinear and smooth Hamiltonians, Theorem 6.4.11 in \cite{C} tells us that if $g\in C^k({\mathbb R^n})$, $k\ge 2$, then $u\in C^k({\mathbb R^n}\times \{t>0\}\setminus \overline \Sigma)$. But  next simple examples show that it does not hold for the Eikonal equation.

{\bf Example 1: } {\it If $g=y^2$, $y\in {\mathbb R}$, then $u(x, t)=0$ when $|x|\le t$, $u(x, t)=g(x-t)=(x-t)^2$ when $|x|> t$.

It is easy to see that $u\in C^2(\{(x, t)||x|\neq t\})$ and is not $C^2$ on lines $|x|=t$. Since there is no nondifferentiable points of $u$,  $\overline \Sigma=\emptyset$. Hence  $u\in C^2({\mathbb R}\times \{t>0\}\setminus \overline \Sigma)=C^2({\mathbb R}\times \{t>0\})$ does not hold.}

{\bf Example 2: } {\it If $g=y^3$, then $u$ is in $C^3({\mathbb R}\times \{t>0\}) $.}

{\bf Example 3: } {\it If $g=y^4$, then $u$ is in $C^2({\mathbb R}\times \{t>0\}) $, but not in $C^4({\mathbb R}\times \{t>0\}) $.}

From example 1, we can see that when   $C^k$ regularity is considered, points on  characteristics  from $0$ with $Dg(0)=0$ and $|P|=1$  are not good points from the starting time, instead of the termination time. From example  3, points on characteristics  from $0$ with $Dg(0)=0$ and $|P|=1$  are good points  from $t=0$ to $\infty$. Hence the $C^k$ regularity remains uncertain for characteristics from some $y_0$ with $Dg(y_0)=0$ and $|P|=1$. If we cut out these  characteristics from $t=0$ to $t_s(y_0, P)$; i.e. these effective characteristics, then the situation is clear. Hence the following main theorem about $C^k$ regularity of $u$ holds. First let $$P^0= \{(x, t)|x=y_0+Pt, t\in [0, t_s(y_0, P)], 
Dg(y_0)=0,
|P|=1\}).$$
\begin{theorem}\label{C2}For any point $(x, t)\in{\mathbb R^n}\times \{t>0\}\setminus (\Sigma\cup T_1\cup P^0)$, there exists a neighborhood $U$ of $(x, t)$ such that $u\in C^k(U)$.
\end{theorem}
\begin{proof}For any point $(x, t)$ not in $\Sigma\cup T_1\cup P^0$, if $(x, t)$ has a minimizer $y_0$ with $Dg(y_0)\neq 0$, since the set of termination points from points with $Dg\neq 0$ $T(M)=\Sigma\cup T_1$, then $(x, t)$ is not in $T(M)$, hence $t < \overline t_s(y_0)$. By Corollary \ref{2}, the result holds.

If all minimizers $y_0$ for  $(x, t)$ have  $Dg(y_0) = 0$, since $(x, t)$  is not in $P^0$, $(x, t)$ is on the characteristic from $y_0$ with the direction $P$ such that $|P| < 1$. By 3) in Proposition \ref{P<1}, we have that $t < t_s(y_0, P)$, then by 2) in Proposition \ref{P<1}, the result is proved.
\end{proof}
We have known that $\Sigma\cup T_1\supset \Sigma$ and $\overline \Sigma \supset \Sigma\cup T_1$, from the theorem above, it can be immediately obtained that for $C^2$ initial data, the following holds:
\begin{corollary}$(\Sigma\cup T_1\cup P^0)\supset \overline \Sigma $.
\end{corollary}



 Tian-Hong Li, 
 {\sl thli@math.ac.cn}\\

 JingHua Wang, 
 {\sl jwang@amss.ac.cn  }\\

    HaiRui Wen,
{\sl  wenhr@bit.edu.cn}

\end{document}